\documentclass[a4paper,12pt]{amsart}

\usepackage[utf8]{inputenc}
\usepackage{amsfonts,amssymb}
\usepackage{amsmath}
\usepackage{graphicx}

\usepackage{comment}
\usepackage{color}
\usepackage{tikz}
\usetikzlibrary{topaths}
\usetikzlibrary{decorations.markings}
\tikzset{->-/.style={decoration={
  markings,
  mark=at position 0.5 with {\arrow{>}}},postaction={decorate}}}
\tikzset{-<-/.style={decoration={
  markings,
  mark=at position 0.5 with {\arrow{<}}},postaction={decorate}}}

\usepackage{a4wide}

\usepackage{empheq}
\numberwithin{equation}{section}

\usepackage{amsthm}
\theoremstyle{plain}
\newtheorem{theorem}{Theorem}[section]
\newtheorem{lemma}[theorem]{Lemma}
\newtheorem{corollary}[theorem]{Corollary}
\newtheorem{proposition}[theorem]{Proposition}
\newtheorem{observation}[theorem]{Observation}
\newtheorem{cit}[theorem]{Citation}

\newtheorem*{lemma*}{Lemma}
\newtheorem*{corollary*}{Corollary}

\newtheorem*{main:thrm1}{Theorem~A}
\newtheorem*{main:thrm2}{Theorem~B}

\theoremstyle{definition}
\newtheorem{definition}[theorem]{Definition}
\newtheorem{remark}[theorem]{Remark}

\newcommand{\R}{\mathbb{R}}

\newcommand{\Z}{\mathbb{Z}}

\newcommand{\defeq}{\mathrel{\mathop{:}}=}

\newcommand{\lk}{\operatorname{lk}}
\newcommand{\st}{\operatorname{st}}
\newcommand{\Hom}{\operatorname{Hom}}
\newcommand{\Aut}{\operatorname{Aut}}
\newcommand{\Out}{\operatorname{Out}}
\newcommand{\CAT}{\operatorname{CAT}}
\newcommand{\cpx}{\Sigma K}
\newcommand{\forests}{\mathcal{F}}
\newcommand{\aforests}{\mathcal{F}^\uparrow}
\newcommand{\ie}{\mathcal{I}}
\newcommand{\aie}{\mathcal{I}^\uparrow}

\newcommand{\asta}{\st^\uparrow}
\newcommand{\alk}{\lk^\uparrow}
\newcommand{\dlk}{\lk_d}
\newcommand{\ulk}{\lk_u}
\newcommand{\adlk}{\lk_d^\uparrow}
\newcommand{\aulk}{\lk_u^\uparrow}

\DeclareMathOperator{\GL}{GL}

\DeclareMathOperator{\F}{F}
\DeclareMathOperator{\adj}{adj}
\DeclareMathOperator{\proj}{proj}
\DeclareMathOperator{\group}{\Sigma Aut}
\DeclareMathOperator{\pgroup}{P\!\group}

\numberwithin{equation}{section}

\begin{document}

\title[Symmetric automorphisms, BNSR-invariants, and finiteness properties]{Symmetric automorphisms of free groups, BNSR-invariants, and finiteness properties}

\date{\today}


\author{Matthew C.~B.~Zaremsky}
\address{Department of Mathematical Sciences, Binghamton University, Binghamton, NY 13902}
\email{zaremsky@math.binghamton.edu}

\begin{abstract}
 The BNSR-invariants of a group $G$ are a sequence $\Sigma^1(G)\supseteq \Sigma^2(G) \supseteq \cdots$ of geometric invariants that reveal important information about finiteness properties of certain subgroups of $G$. We consider the symmetric automorphism group $\group_n$ and pure symmetric automorphism group $\pgroup_n$ of the free group $F_n$, and inspect their BNSR-invariants. We prove that for $n\ge 2$, all the ``positive'' and ``negative'' character classes of $\pgroup_n$ lie in $\Sigma^{n-2}(\pgroup_n)\setminus \Sigma^{n-1}(\pgroup_n)$. We use this to prove that for $n\ge 2$, $\Sigma^{n-2}(\group_n)$ equals the full character sphere $S^0$ of $\group_n$ but $\Sigma^{n-1}(\group_n)$ is empty, so in particular the commutator subgroup $\group_n'$ is of type $\F_{n-2}$ but not $\F_{n-1}$. Our techniques involve applying Morse theory to the complex of symmetric marked cactus graphs.
\end{abstract}

\maketitle
\thispagestyle{empty}


\section*{Introduction}

The Bieri--Neumann--Strebel--Renz (BNSR) invariants $\Sigma^m(G)$ of a group $G$ are a sequence of geometric invariants $\Sigma^1(G)\supseteq \Sigma^2(G)\supseteq \cdots$ that encode a large amount of information about the subgroups of $G$ containing the commutator subgroup $G'$. For example if $G$ is of type $\F_n$ and $m\le n$ then $\Sigma^m(G)$ reveals precisely which such subgroups are of type $\F_m$. Recall that a group is of \emph{type $F_n$} if it admits a classifying space with compact $n$-skeleton; these \emph{finiteness properties} are an important class of quasi-isometry invariants of groups. The BNSR-invariants are in general very difficult to compute; a complete description is known for the class of right-angled Artin groups \cite{meier98,bux99}, but not many other substantial families of groups. A complete picture also exists for the generalized Thompson groups $F_{n,\infty}$ \cite{bieri10,kochloukova12,zaremsky15}, and the first invariant $\Sigma^1$ is also known for some additional classes of groups, e.g., one-relator groups \cite{brown87bns}, pure braid groups \cite{koban15} and pure symmetric automorphism groups of right-angled Artin groups \cite{orlandi-korner00,koban14}, among others.

In this paper, we focus on the groups $\group_n$ and $\pgroup_n$ of symmetric and pure symmetric automorphisms of the free group $F_n$. An automorphism of $F_n$ is \emph{symmetric} if it takes each basis element to a conjugate of a basis element, and \emph{pure symmetric} if it takes each basis element to a conjugate of itself. These are also known as the \emph{(pure) loop braid groups}, and are the groups of motions of $n$ unknotted unlinked oriented loops in $3$-space; an element describes these loops moving around and through each other, ending up back where they started, either individually in the pure case or just as a set in the non-pure case (but preserving orientation in both cases). Other names for these and closely related groups include welded braid groups, permutation-conjugacy automorphism groups, braid-permutation groups and more. See \cite{damiani16} for a discussion of the many guises of these groups. Some topological properties known for $\pgroup_n$ include that is has cohomological dimension $n-1$ \cite{collins89}, it is a duality group \cite{brady01} and its cohomology ring has been computed \cite{jensen06}.

The first invariant $\Sigma^1(\pgroup_n)$ was fully computed by Orlandi--Korner \cite{orlandi-korner00} (she denotes the group by $P\Sigma_n$). Koban and Piggott subsequently computed $\Sigma^1(G)$ for $G$ the group of pure symmetric automorphisms of any right-angled Artin group \cite{koban14}. One reason that the question of BNSR-invariants is interesting for $\pgroup_n$ is that $\pgroup_n$ is similar to a right-angled Artin group, for instance it admits a presentation in which the relations are all commutators (see Section~\ref{sec:gps}), but for $n\ge 3$ it is not a right-angled Artin group \cite{koban14}, and it is not known whether it is a $\CAT(0)$ group \cite[Question~6.4]{brady01}. The BNSR-invariants are completely known for right-angled Artin groups, but the Morse theoretic proof of this fact in \cite{bux99} made essential use of the $\CAT(0)$ geometry of the relevant complexes.

Our approach here is to use Morse theory applied to the complex of \emph{symmetric marked cactus graphs} $\cpx_n$ to prove the following main results:

\begin{main:thrm1}
 For $n\ge 2$, if $\chi$ is a positive or negative\footnote{For the definitions of ``positive'' and ``negative'' consult Definition~\ref{def:pos_supp}.} character of $\pgroup_n$ then $[\chi]\in\Sigma^{n-2}(\pgroup_n)\setminus\Sigma^{n-1}(\pgroup_n)$.
\end{main:thrm1}

\begin{main:thrm2}
 For $n\ge 2$, we have $\Sigma^{n-2}(\group_n) = S(\group_n) = S^0$ and $\Sigma^{n-1}(\group_n) = \emptyset$. In particular the commutator subgroup $\group_n'$ is of type $\F_{n-2}$ but not $\F_{n-1}$.
\end{main:thrm2}

For example this shows that $\group_n'$ is finitely generated if and only if $n\ge 3$, and finitely presentable if and only if $n\ge 4$. It appears that these are already new results (except for the fact that $\group_2'$ is not finitely generated, which is easy to see since $\group_2 \cong F_2\rtimes S_2$). Theorem~B also provides what could be viewed as the first examples for $m\ge 2$ of ``naturally occurring'' groups $G$ of type $\F_\infty$ such that $\Sigma^{m-1}(G)=S(G)$ but $\Sigma^m(G)=\emptyset$, and of groups of type $\F_\infty$ whose commutator subgroups have arbitrary finiteness properties. (One can also construct more \emph{ad hoc} examples: we have noticed that taking a semidirect product of $F_2^n$ with the Coxeter group of type $B_n=C_n$ also produces a group with these properties.) As a remark, contrasting the loop braid group $\group_n$ with the classical braid group $B_n$, it is easy to see that $\Sigma^m(B_n)=S(B_n)=S^0$ for all $m$ and $n$, and $B_n'$ is of type $\F_\infty$ for all $n$.

For the case $n=3$ we can actually get a full computation of $\Sigma^m(\pgroup_3)$; Orlandi--Korner already computed $\Sigma^1(\pgroup_3)$ (it is dense in $S(\pgroup_3)$; see Citation~\ref{cit:sig1}), and we prove that $\Sigma^2(\pgroup_3)=\emptyset$ (see Theorem~\ref{thrm:n3}). We tentatively conjecture that $\Sigma^{n-2}(\pgroup_n)$ is always dense in $S(\pgroup_n)$ and $\Sigma^{n-1}(\pgroup_n)$ is always empty, but for $n\ge 4$ it seems this cannot be proved using our techniques, as discussed in Remark~\ref{rmk:n>3}.

As a remark, there is a result of Pettet involving finiteness properties of some other normal subgroups of $\pgroup_n$. Namely she found that the kernel of the natural projection $\pgroup_n \to \pgroup_{n-1}$ is finitely generated but not finitely presentable when $n\ge 3$ \cite{pettet10}. This is in contrast to the pure braid situation, where the kernel of the ``forget a strand'' map $PB_n \to PB_{n-1}$ is of type $\F_\infty$ (in fact it is the free group $F_{n-1}$).

This paper is organized as follows. In Section~\ref{sec:invs_morse} we recall the background on BNSR-invariants and Morse theory. In Section~\ref{sec:gps} we discuss the groups of interest, and in Section~\ref{sec:cpx} we discuss the complex $\cpx_n$. We prove Theorem~A in Section~\ref{sec:asc_link}, and with Theorem~A in hand we quickly prove Theorem~B in Section~\ref{sec:symm_auts}.

\subsection*{Acknowledgments} I am grateful to Alex Schaefer for a helpful conversation about graph theory that in particular helped me figure out how to prove that $\Sigma^2(\pgroup_3)=\emptyset$ (see Subsection~\ref{sec:n3}), to Celeste Damiani for pointing me toward the paper \cite{savushkina96}, and to Robert Bieri for enlightening discussions about the novelty of the behavior of the BNSR-invariants found in Theorem~B.


\section{BNSR-invariants and Morse theory}\label{sec:invs_morse}

In this rather technical section we recall the definition of the BNSR-invariants, and set up the Morse theoretic approach that we will use. The results in Subsection~\ref{sec:bnsr_morse} are general enough that we expect they should be useful in the future to compute BNSR-invariants of other interesting groups.

\subsection{BNSR-invariants}\label{sec:invs}

A CW-complex $Z$ is called a \emph{classifying space} for $G$, or $K(G,1)$, if $\pi_1(Z)\cong G$ and $\pi_k(Z)=0$ for all $k\ne 1$. We say that $G$ is of \emph{type $F_n$} if it admits a $K(G,1)$ with compact $n$-skeleton. For example $G$ is of type $\F_1$ if and only if it is finitely generated and of type $\F_2$ if and only if it is finitely presentable. If $G$ is of type $\F_n$ for all $n$ we say it is of \emph{type $F_\infty$}. If $G$ acts properly and cocompactly on an $(n-1)$-connected CW-complex, then $G$ is of type $\F_n$.

\begin{definition}[BNSR-invariants]\label{def:sig_invs}
 Let $G$ be a group acting properly and cocompactly on an $(n-1)$-connected CW-complex $Y$ (so $G$ is of type $\F_n$). Let $\chi\colon G\to\R$ be a \emph{character} of $G$, i.e., a homomorphism to $\R$. There exists a map $h_\chi \colon Y \to \R$, which we will call a \emph{character height function}, such that $h_\chi(g.y)=\chi(g)+h_\chi(y)$ for all $g\in G$ and $y\in Y$. For $t\in\R$ let $Y_{\chi\ge t}$ be the full subcomplex of $Y$ supported on those $0$-cells $y$ with $h_\chi(y)\ge t$. Let $[\chi]$ be the equivalence class of $\chi$ under scaling by positive real numbers. The \emph{character sphere} $S(G)$ is the set of non-trivial character classes $[\chi]$. For $m\le n$, the $m$th \emph{BNSR-invariant} $\Sigma^m(G)$ is defined to be
 $$\Sigma^m(G)\defeq \{[\chi]\in S(G)\mid (Y_{\chi\ge t})_{t\in\R} \text{ is essentially $(m-1)$-connected}\}\text{.}$$
\end{definition}

Recall that $(Y_{\chi\ge t})_{t\in\R}$ is said to be \emph{essentially $(m-1)$-connected} if for all $t\in\R$ there exists $-\infty<s\le t$ such that the inclusion of $Y_{\chi\ge t}$ into $Y_{\chi\ge s}$ induces the trivial map in $\pi_k$ for all $k\le m-1$.

It turns out $\Sigma^m(G)$ is well defined up to the choice of $Y$ and $h_\chi$ (see for example \cite[Definition~8.1]{bux04}). As a remark, the definition there used the filtration by sets $h_\chi^{-1}([t,\infty))_{t\in\R}$, but thanks to cocompactness this filtration is essentially $(m-1)$-connected if and only if our filtration $(Y_{\chi\ge t})_{t\in\R}$ is.

One important application of BNSR-invariants is the following:

\begin{cit}\cite[Theorem~B and Remark~6.5]{bieri88}\label{cit:sig_fin}
 Let $G$ be a group of type $\F_m$. Let $G'\le H\le G$. Then $H$ is of type $\F_m$ if and only if for every non-trivial character $\chi$ of $G$ such that $\chi(H)=0$, we have $[\chi]\in\Sigma^m(G)$.
\end{cit}

For example, if $H=\ker(\chi)$ for $\chi$ a \emph{discrete} character of $G$, i.e., one with image $\Z$, then $H$ is of type $\F_m$ if and only if $[\pm\chi]\in\Sigma^m(G)$. Also note that $G'$ itself is of type $\F_m$ if and only if $\Sigma^m(G)=S(G)$.

Other important classical properties of the $\Sigma^m(G)$ are that they are all open subsets of $S(G)$ and that they are invariant under the natural action of $\Aut(G)$ on $S(G)$ \cite{bieri87,bieri88}.

\subsection{Morse theory}\label{sec:morse}

Bestvina--Brady Morse theory can be a useful tool for computing BNSR-invariants. In this section we give the relevant definitions and results from Morse theory, in the current level of generality needed.

Let $Y$ be an affine cell complex (see \cite[Definition~2.1]{bestvina97}). The \emph{star} $\st_Y v$ of a $0$-cell $v$ in $Y$ is the subcomplex of $Y$ consisting of cells that are faces of cells containing $v$. The \emph{link} $\lk_Y v$ of $v$ is the simplicial complex $\st_Y v$ of directions out of $v$ into $\st_Y v$. We will suppress the subscript $Y$ from the notation when it is clear from context. If $v$ and $w$ are distinct $0$-cells sharing a $1$-cell we will call $v$ and $w$ \emph{adjacent} and write $v \adj w$.

In \cite{bestvina97}, Bestvina and Brady defined a \emph{Morse function} on an affine cell complex $Y$ to be a map $Y\to \R$ that is affine on cells, takes discretely many values on the $0$-cells, and is non-constant on $1$-cells. When using Morse theory to compute BNSR-invariants though, these last two conditions are often too restrictive. The definition of \emph{Morse function} that will prove useful for our purposes is as follows.

\begin{definition}[Morse function]\label{def:ht_fxn}
 Let $Y$ be an affine cell complex and let $h\colon Y \to \R$ and $f\colon Y \to \R$ be functions that are affine on cells. We call $(h,f)\colon Y \to \R\times \R$ a \emph{Morse function} if the set $\{h(v)-h(w)\mid v,w\in Y^{(0)}\text{, } v\adj w\}$ does not have $0$ as a limit point (we will call it \emph{discrete near $0$}), the set $\{f(v)\mid v\in Y^{(0)}\}$ is finite\footnote{In what follows it will be clear that ``finite'' could be replaced with ``well ordered'' but for our present purposes we will just assume it is finite.}, and if $v,w\in Y^{(0)}$ with $v\adj w$ and $h(v)=h(w)$ then $f(v)\ne f(w)$.
\end{definition}

For example if $h$ takes discrete values on $0$-cells and distinct values on adjacent $0$-cells, then (taking $f$ to be constant and ignoring it) we recover Bestvina and Brady's notion of ``Morse function''.

Using the usual order on $\R$ and the lexicographic order on $\R\times \R$, it makes sense to compare $(h,f)$ values of $0$-cells. On a given cell $c$, since $h$ and $f$ are affine on $c$ it is clear that $(h,f)$ achieves its maximum and minimum values at unique faces of $c$, and the last assumption in Definition~\ref{def:ht_fxn} ensures these will be $0$-cells.

\begin{definition}[Ascending star/link]\label{def:asc_lk}
 Given a Morse function $(h,f)$ on an affine cell complex $Y$, define the \emph{ascending star} $\asta v$ of a $0$-cell $v$ in $Y$ to be the subcomplex of $\st v$ consisting of all faces of those cells $c$ for which the unique $0$-face of $c$ where $(h,f)$ achieves its minimum is $v$. Define the \emph{ascending link} $\alk v$ to be the subcomplex of $\lk v$ consisting of directions into $\asta v$. Note that $\alk v$ is a full subcomplex of $\lk v$, since $h$ and $f$ are affine on cells.
\end{definition}

For $Y$ an affine cell complex, $(h,f)$ a Morse function on $Y$ and $t\in\R$, denote by $Y_{h\ge t}$ the subcomplex of $Y$ supported on those $0$-cells $v$ with $h(v)\ge t$.

\begin{lemma}[Morse Lemma]\label{lem:morse}
 Let $Y$ be an affine cell complex and $(h,f) \colon Y \to \R\times\R$ a Morse function. Let $t\in\R$ and $s\in [-\infty,t)$. If for all $0$-cells $v$ with $h(v)\in [s,t)$ the ascending link $\alk v$ is $(m-1)$-connected, then the inclusion $Y_{h\ge t} \to Y_{h\ge s}$ induces an isomorphism in $\pi_k$ for $k\le m-1$ and an epimorphism in $\pi_m$.
\end{lemma}

\begin{proof} The essential parts of the proof are the same as in \cite{bestvina97}. Choose $\varepsilon>0$ such that for any $v \adj w$, $|h(v)-h(w)|\not\in (0,\varepsilon)$ (this is possible since the set of values $h(v)-h(w)$ for $v \adj w$ is discrete near $0$). We can assume by induction (and by compactness of spheres if $s=-\infty$) that $t-s\le \varepsilon$. In particular if adjacent $0$-cells $v$ and $w$ both lie in $Y_{h\ge s}\setminus Y_{h\ge t}$, then $h(v)=h(w)$ and $f(v)\ne f(w)$. To build up from $Y_{h\ge t}$ to $Y_{h\ge s}$, we need to glue in the $0$-cells of $Y_{h\ge s}\setminus Y_{h\ge t}$ along their relative links in some order such that upon gluing in $v$, all of $\alk v$ is already present, but nothing else in $\lk v$, so the relative link is precisely the ascending link. To do this, we put any order we like on each set $F_i \defeq \{v\in Y_{h\ge s}^{(0)}\setminus Y_{h\ge t}^{(0)}\mid f(v)=i\}$ for $i\in f(Y^{(0)})$, and then extend these to an order on $Y_{h\ge s}^{(0)}\setminus Y_{h\ge t}^{(0)}$ by declaring that everything in $F_i$ comes after everything in $F_j$ whenever $i<j$. Now when we glue in $v$, for $w\in \lk v$ we have $w\in \alk v$ if and only if either $h(w)>h(v)$, in which case $h(w)\ge t$ and $w$ is already present, or $h(w)=h(v)$ and $f(w)>f(v)$, in which case $w\in F_{f(w)}$ is also already present. Since the relevant ascending links are $(m-1)$-connected by assumption, the result follows from the Seifert--van Kampen, Mayer--Vietoris and Hurewicz Theorems.
\end{proof}

As a corollary to the proof, we have:

\begin{corollary}\label{cor:bad_spheres_stay}
 With the same setup as the Morse Lemma, if additionally for all $0$-cells $v$ with $s\le h(v)<t$ we have $\widetilde{H}_{m+1}(\alk v)=0$, then the inclusion $Y_{h\ge t} \to Y_{h\ge s}$ induces an injection in $\widetilde{H}_{m+1}$.
\end{corollary}

\begin{proof}
 In the proof of the Morse Lemma, we saw that $Y_{h\ge s}$ is obtained from $Y_{h\ge t}$ by coning off the ascending links of $0$-cells $v$ with $s\le h(v)<t$, so this is immediate from the Mayer--Vietoris sequence.
\end{proof}

For example if $Y$ is $(m+1)$-dimensional, so the links are at most $m$-dimensional, then this additional condition will always be satisfied.

\medskip

Wen dealing with BNSR-invariants, the following is particularly useful:

\begin{corollary}\label{cor:morse}
 Let $Y$ be an $(m-1)$-connected affine cell complex with a Morse function $(h,f)$. Suppose there exists $q$ such that, for every $0$-cell $v$ of $Y$ with $h(v)<q$, $\alk v$ is $(m-1)$-connected. Then the filtration $(Y_{h\ge t})_{t\in\R}$ is essentially $(m-1)$-connected. Now assume additionally that $\widetilde{H}_{m+1}(Y)=0$ and for every $0$-cell $v$ of $Y$ with $h(v)<q$, $\widetilde{H}_{m+1}(\alk v)=0$, and that for all $p$ there exists a $0$-cell $v$ with $h(v)< p$ such that $\widetilde{H}_m(\alk v)\ne 0$. Then the filtration $(Y_{h\ge t})_{t\in\R}$ is not essentially $m$-connected.
\end{corollary}

\begin{proof}
 By the Morse Lemma, for any $r\le q$ the inclusion $Y_{h\ge r} \to Y=Y_{h\ge-\infty}$ induces an isomorphism in $\pi_k$ for $k\le m-1$. Since $Y$ is $(m-1)$-connected, so is $Y_{h\ge r}$. Now for any $t\in\R$ we just need to choose $s=\min\{q,t\}$ and we get that the inclusion $Y_{h\ge t}\to Y_{h\ge s}$ induces the trivial map in $\pi_k$ for $k\le m-1$, simply because $Y_{h\ge s}$ is $(m-1)$-connected.

 For the second claim, suppose that $(Y_{h\ge t})_{t\in\R}$ is essentially $m$-connected. Say $t< q$, and choose $s\le t$ such that the inclusion $Y_{h\ge t} \to Y_{h\ge s}$ induces the trivial map in $\pi_k$ for $k\le m$. Also, since $t< q$, this inclusion induces a surjection in these $\pi_k$ by the Morse Lemma, so in fact $Y_{h\ge s}$ itself is $m$-connected, as are all $Y_{h\ge r}$ for $r\le s$ (for the same reason). Now choose $v$ such that $h(v)< s$ and $\widetilde{H}_m(\alk v)\ne 0$. Since $\widetilde{H}_m(Y_{h\ge r})=0$ for all $r\le s$, Mayer--Vietoris and Corollary~\ref{cor:bad_spheres_stay} say that $\widetilde{H}_{m+1}(Y_{h\ge q})\ne 0$ for any $q\le h(v)$. But this includes $q=-\infty$, which contradicts our assumption that $\widetilde{H}_{m+1}(Y)=0$.
\end{proof}

\subsection{BNSR-invariants via Morse theory}\label{sec:bnsr_morse}

We now return to the situation in Definition~\ref{def:sig_invs}, so $Y$ is an $(n-1)$-connected CW-complex on which $G$ acts properly and cocompactly (and, we assume, cellularly), $\chi$ is a character of $G$, and $h_\chi$ is a character height function on $Y$. The goal of this subsection is to establish a Morse function on $Y$ using $h_\chi$.

Let us make two additional assumptions. First, assume $Y$ is simplicial (this is just to ensure that any function on $Y^{(0)}$ can be extended to a function on $Y$ that is affine on cells). Second, assume that no adjacent $0$-simplicies in $Y$ share a $G$-orbit (if this is not the case, it can be achieved by subdividing). Let $\overline{f} \colon Y^{(0)}/G \to \R$ be any function that takes distinct values on adjacent $0$-cells, where the cell structure on $Y/G$ is induced from $Y$. (Just to give some examples, one could construct $\overline{f}$ by randomly assigning distinct values to the $0$-cells in $Y/G$, or one could take the barycentric subdivision and have $\overline{f}$ read the dimension.) Define $f\colon Y^{(0)}\to \R$ via $f(v)\defeq \overline{f}(G.v)$, and extend $f$ to a map (also called $f$) on all of $Y$ by extending affinely to each simplex.

\begin{lemma}\label{lem:bnsr_morse}
 With $Y$, $h_\chi$ and $f$ as above, $(h_\chi,f)\colon Y \to \R\times \R$ is a Morse function.
\end{lemma}

\begin{proof}
 The functions $h_\chi$ and $f$ are affine on cells by construction. The set $\{f(v)\mid v\in Y^{(0)}\}$ equals the set $\{\overline{f}(G.v)\mid G.v\in Y^{(0)}/G\}$, which is finite since $Y/G$ is compact. For any $g\in G$ we have $h_\chi(g.v)-h_\chi(g.w) = \chi(g)+h_\chi(v)-\chi(g)-h_\chi(w) = h_\chi(v)-h_\chi(w)$, so by compactness of $Y/G$ the set $\{h_\chi(v)-h_\chi(w)\mid v,w\in Y^{(0)}\text{, } v\adj w\}$ is finite (and hence discrete near $0$). Finally, since $\overline{f}$ takes distinct values on adjacent $0$-cells in $Y/G$, and no adjacent $0$-cells in $Y$ share an orbit, we see $f$ takes distinct values on adjacent $0$-cells in $Y$.
\end{proof}

In particular Corollary~\ref{cor:morse} can now potentially be used to prove that $(Y_{\chi\ge t})_{t\in\R}$ is or is not essentially $(m-1)$-connected, and hence that $[\chi]$ is or is not in $\Sigma^m(G)$.

While any $f$ constructed as above will make $(h_\chi,f)$ a Morse function, this does not mean every $f$ may be useful, for instance if the ascending links are not as highly connected as one would hope. In fact it seems likely that situations exist where every choice of $f$ yields a ``useless'' Morse function. Hence, in practice one hopes to have a concrete space $Y$ with a natural choice of $f$ that produces nice ascending links.


\section{(Pure) symmetric automorphism groups}\label{sec:gps}

We now turn to our groups of interest.

Let $F_n$ be the free group with basis $S\defeq \{x_1,\dots,x_n\}$. An automorphism $\alpha\in\Aut(F_n)$ is called \emph{symmetric} if for each $i\in [n]\defeq \{1,\dots,n\}$, $(x_i)\alpha$ is conjugate to $x_j$ for some $j$; if each $(x_i)\alpha$ is conjugate to $x_i$ we call $\alpha$ \emph{pure symmetric}\footnote{Automorphisms will be acting on the right here, so we will reflect this in the notation.}. Note that in some texts, ``symmetric'' allows for $(x_i)\alpha$ to be conjugate to some $x_j^{-1}$, but we do not allow that here. Denote by $\group_n$ the group of all symmetric automorphisms of $F_n$, and by $\pgroup_n$ the group of pure symmetric automorphisms. The abelianization $F_n \to \Z^n$ induces a surjection $\Aut(F_n)\to \GL_n(\Z)$, and the restriction of this map to $\group_n$ yields a splitting
$$\group_n \cong \pgroup_n \rtimes S_n \text{.}$$

An equivalent description of $\group_n$ (and $\pgroup_n$) is as the \emph{(pure) loop braid group}, i.e., the group of (pure) motions of $n$ unknotted, unlinked oriented circles in $3$-space. The subgroup of $\group_n$ consisting of those automorphisms taking $x_1\cdots x_n$ to itself is isomorphic to the classical braid group $B_n$, and the intersection of this with $\pgroup_n$ is the classical pure braid group $PB_n$ \cite{savushkina96}. Other names for $\group_n$ and closely related groups include welded braid groups, permutation-conjugacy automorphism groups, braid-permutation groups and more. Details on the various viewpoints for these groups can be found for example in \cite{damiani16}.

In \cite{mccool86}, McCool found a (finite) presentation for $\pgroup_n$. The generators are the automorphisms $\alpha_{i,j}$ ($i\ne j$) given by $(x_i)\alpha_{i,j} = x_j^{-1} x_i x_j$ and $(x_k)\alpha_{i,j} = x_k$ for $k\ne i$, and the defining relations are $[\alpha_{i,j},\alpha_{k,\ell}]=1$, $[\alpha_{i,j},\alpha_{k,j}]=1$ and $[\alpha_{i,j}\alpha_{k,j},\alpha_{i,k}]=1$, for distinct $i,j,k,\ell$. (In particular note that this implies $\pgroup_2\cong F_2$.) It will also be convenient later to consider automorphisms $\alpha_{I,j}$, defined via
$$\alpha_{I,j} \defeq \prod_{i\in I}\alpha_{i,j}$$
for $I\subseteq [n]\setminus \{j\}$, where the product can be taken in any order thanks to the relation $[\alpha_{i,j},\alpha_{k,j}]=1$. Following Collins \cite{collins89} we call these \emph{symmetric Whitehead automorphisms}.

Since the defining relations in McCool's presentation are commutators, we immediately see that $\pgroup_n$ has abelianization $\Z^{n(n-1)}$, with basis $\{\overline{\alpha}_{i,j}\mid i\ne j\}$. Since $S_n$ acts transitively on the $\alpha_{i,j}$, we also quickly compute that $\group_n$ abelianizes to $\Z\times (\Z/2\Z)$ for all $n\ge 2$. A natural basis for the vector space $\Hom(\pgroup_n,\R) \cong \R^{n(n-1)}$ is the dual of $\{\overline{\alpha}_{i,j}\mid i\ne j\}$. This dual basis has a nice description that we will now work up to. For $\alpha\in\pgroup_n$ let $w_{i,\alpha} \in F_n$ be the elements such that $(x_i)\alpha=x_i^{w_{i,\alpha}}$. For each $i\ne j$ define $\chi_{i,j}\colon \pgroup_n \to \Z$ by sending $\alpha$ to $\varphi_j(w_{i,\alpha})$, where $\varphi_j \colon F_n \to \Z$ are the projections sending $x_j$ to $1$ and the other generators to $0$.

\begin{lemma}\label{lem:char_homs}
 Each $\chi_{i,j}$ is a homomorphism.
\end{lemma}

\begin{proof}
 Let $\alpha,\beta\in \pgroup_n$ and $i\in [n]$. Write $w_{i,\alpha}=x_{k_1}^{\varepsilon_1}\cdots x_{k_r}^{\varepsilon_r}$ for $k_1,\dots,k_r\in [n]$ and $\varepsilon_1,\dots,\varepsilon_r\in \{\pm1\}$, so we have
 $$(x_i)\alpha \circ \beta = (x_{k_r}^{-\varepsilon_r})\beta \cdots (x_{k_1}^{-\varepsilon_1})\beta w_{i,\beta}^{-1} x_i w_{i,\beta} (x_{k_1}^{\varepsilon_1})\beta \cdots (x_{k_r}^{\varepsilon_r})\beta \text{.}$$
 In particular
 $$w_{i,\alpha\circ\beta} = w_{i,\beta} (x_{k_1}^{\varepsilon_1})\beta \cdots (x_{k_r}^{\varepsilon_r})\beta \text{.}$$
 Note that $\varphi_j((x_{k_1}^{\varepsilon_1})\beta \cdots (x_{k_r}^{\varepsilon_r})\beta) = \varphi_j(x_{k_1}^{\varepsilon_1} \cdots x_{k_r}^{\varepsilon_r})$, so $\varphi_j(w_{i,\alpha\circ\beta}) = \varphi_j(w_{i,\alpha}) + \varphi_j(w_{i,\beta})$, as desired.
\end{proof}

Clearly $\chi_{i,j}(\alpha_{k,\ell}) = \delta_{(i,j),(k,\ell)}$ (the Kronecker delta), so $\{\chi_{i,j}\mid i\ne j\}$ is the basis of $\Hom(\pgroup_n,\R)$ dual to $\{\overline{\alpha}_{i,j}\mid i\ne j\}$. Since $\Hom(\pgroup_n,\R)\cong \R^{n(n-1)}$ we know that the character sphere $S(\pgroup_n)$ is $S(\pgroup_n)=S^{n(n-1)-1}$.

For the group $\group_n$, $\Hom(\group_n,\R)\cong \R$ for all $n$, so to find a basis we just need a non-trivial character. We know $\group_n = \pgroup_n \rtimes S_n$, so the most natural candidate is the character reading $1$ on each $\alpha_{i,j}$ and $0$ on $S_n$. Note that $S(\group_n)=S^0$ for all $n\ge 2$.

Writing an arbitrary character of $\pgroup_n$ as $\chi=\sum_{i\ne j} a_{i,j}\chi_{i,j}$, we recall Orlandi--Korner's computation of $\Sigma^1(\pgroup_n)$:

\begin{cit}\cite{orlandi-korner00}\label{cit:sig1}
 We have $[\chi]\in\Sigma^1(\pgroup_n)$ unless either
 \begin{enumerate}
  \item there exist distinct $i$ and $j$ such that $a_{p,q}=0$ whenever $\{p,q\} \not\subseteq \{i,j\}$, or
  \item there exist distinct $i$, $j$ and $k$ such that $a_{p,q}=0$ whenever $\{p,q\} \not\subseteq \{i,j,k\}$ and moreover $a_{p,q}=-a_{p',q}$ whenever $\{p,p',q\}=\{i,j,k\}$.
 \end{enumerate}
 In these cases $[\chi]\not\in\Sigma^1(\pgroup_n)$.
\end{cit}

For example, $\Sigma^1(\pgroup_2)$ is empty (which we know anyway since $\pgroup_2\cong F_2$) and $\Sigma^1(\pgroup_3)$ is a $5$-sphere with three $1$-spheres and one $2$-sphere removed, so in particular $\Sigma^1(\pgroup_3)$ is dense in $S(\pgroup_3)$.

\medskip

The groups $\group_n$ and $\pgroup_n$ are of type $\F_\infty$ (this can be seen for example after work of Collins \cite{collins89})\footnote{Actually, $\pgroup_n$ is even of ``type $\F$'', meaning it has a compact classifying space, but we will not need this fact.}, so one can ask what $\Sigma^m(\group_n)$ and $\Sigma^m(\pgroup_n)$ are, for any $m$ and $n$. One thing we know, which we will use later, is that the invariants are all closed under taking antipodes:

\begin{observation}\label{obs:antipodes}
 If $[\chi]\in\Sigma^m(G)$ for $G=\group_n$ or $\pgroup_n$ then $[-\chi]\in\Sigma^m(G)$.
\end{observation}

\begin{proof}
 The automorphism of $F_n$ taking each $x_i$ to $x_i^{-1}$ induces an automorphism of $\group_n$ and $\pgroup_n$ under which each character $\chi$ maps to $-\chi$. The result now follows since $\Sigma^m(G)$ is invariant under $\Aut(G)$.
\end{proof}

We can now state our main results.

\begin{definition}[Positive/negative character]\label{def:pos_supp}
 Call $\chi=\sum_{i\ne j} a_{i,j} \chi_{i,j}$ \emph{positive} if $a_{i,j}>0$ for all $i,j$, and \emph{negative} if $a_{i,j}<0$ for all $i,j$.
\end{definition}

\begin{main:thrm1}\label{thrm:main1}
 For $n\ge 2$, if $\chi$ is a positive or negative character of $\pgroup_n$ then $[\chi]\in\Sigma^{n-2}(\pgroup_n)\setminus\Sigma^{n-1}(\pgroup_n)$.
\end{main:thrm1}

As a remark, thanks to Observation~\ref{obs:antipodes} the negative character classes lie in a given $\Sigma^m(\pgroup_n)$ if and only if the positive ones do, so we only need to prove Theorem~A for positive characters.

\begin{main:thrm2}\label{thrm:main2}
 For $n\ge 2$, we have $\Sigma^{n-2}(\group_n) = S(\group_n) = S^0$ and $\Sigma^{n-1}(\group_n) = \emptyset$. In particular the commutator subgroup $\group_n'$ is of type $\F_{n-2}$ but not $\F_{n-1}$.
\end{main:thrm2}

The commutator subgroups $\group_n'$ and $\pgroup_n'$ are easy to describe; see for example Lemmas~4 and~5 of \cite{savushkina96}. The commutator subgroup $\pgroup_n'$ consists of those automorphisms taking each $x_i$ to $w_i^{-1} x_i w_i$ for some $w_i\in F_n'$. In other words, $\pgroup_n'$ is just the intersection of all the $\ker(\chi_{i,j})$. Note that for $n\ge 2$ the commutator subgroup $\pgroup_n'$ is not finitely generated, since it surjects onto $\pgroup_2'\cong F_2'$. The commutator subgroup $\group_n'$ consists of those automorphisms taking $x_i$ to $w_i^{-1} x_{\pi(i)} w_i$ for some even permutation $\pi\in S_n$ (i.e., $\pi\in A_n$) and satisfying $\varphi(w_1\cdots w_n)=0$ where $\varphi \colon F_n \to \Z$ is the map taking each basis element to $1$. As a remark, the abelianization map $\group_n \to \Z$ splits, for instance by sending $\Z$ to $\langle \alpha_{1,2}\rangle$, so we have $\group_n = \group_n' \rtimes \Z$.

In Section~\ref{sec:symm_auts} we will be able to deduce Theorem~B from Theorem~A quickly by using the next lemma. If we write $BB_n$ for the kernel of the character $\sum_{i\ne j}\chi_{i,j}$ of $\pgroup_n$ taking each $\alpha_{i,j}$ to $1$, so $BB_n$ is the ``Bestvina--Brady-esque'' subgroup of $\pgroup_n$, then we have:

\begin{lemma}\label{lem:bb_decomp}
 $\group_n'=BB_n \rtimes A_n$.
\end{lemma}

\begin{proof}
 When we restrict the map $\group_n \to S_n$ to $\group_n'$, by the above description we know that the image is $A_n$. This map splits, and the kernel of this restricted map is the kernel of the original map, which is $\pgroup_n$, intersected with $\group_n'$. The above description tells us that this consists of all pure symmetric automorphisms $\alpha$ such that $\varphi(w_{1,\alpha}\cdots w_{n,\alpha})=0$, and from the definition of the $\chi_{i,j}$ it is clear that $\varphi(w_{1,\alpha}\cdots w_{n,\alpha})=\sum_{i\ne j} \chi_{i,j}(\alpha)$, so we are done.
\end{proof}

In particular $BB_n$ has finite index in $\group_n'$, so they have the same finiteness properties.

To prove Theorems~A and~B we need a complex on which the groups act nicely, and to understand ascending links. We discuss the complex in Section~\ref{sec:cpx} and the ascending links in Section~\ref{sec:asc_link}.


\section{The complex of symmetric marked cactus graphs}\label{sec:cpx}

In \cite{collins89}, Collins found a contractible simplicial complex $\cpx_n$ on which the ``Outer'' versions of $\group_n$ and $\pgroup_n$ act properly and cocompactly, described by \emph{symmetric marked cactus graphs}. We will use the obvious analog of this complex for our groups. Thanks to the action being proper and cocompact, and the complex being contractible, it can be used to ``reveal'' the BNSR-invariants of the groups, as per Definition~\ref{def:sig_invs}. In this section we recall the construction of $\cpx_n$ and set up the character height functions that will then be used in the following sections to prove our main results.

\textbf{Terminology:} By a \emph{graph} we will always mean a connected finite directed graph with one vertex specified as the \emph{basepoint}, such that the basepoint has degree at least two and all other vertices have degree at least three. We will use the usual terminology of initial and terminal endpoints of an edge, paths, cycles, reduced paths, simple cycles, subtrees, subforests and spanning trees.

Our graphs will always be understood to have rank $n$, unless otherwise specified.

Let $R_n$ be the $n$-petaled rose, that is the graph with one vertex $*$ (which is necessarily the basepoint) and $n$ edges. Then $\pi_1(R_n)\cong F_n$, and we identify $\Aut(F_n)$ with the group of basepoint-preserving self-homotopy equivalences of $R_n$, modulo homotopy.

\begin{definition}[Cactus graph, cladode, base, above, projection, before/after, between]\label{def:cactus}
 A graph $\Gamma$ is called a \emph{cactus graph} if every edge is contained in precisely one simple cycle. We will refer to the simple cycles, viewed as subgraphs, as \emph{cladodes}. For example the petals of the rose $R_n$ are precisely its cladodes. We will assume the orientations of the edges are such that each cladode is a directed cycle, that is, no distinct edges of a cladode share an origin (or terminus). If $C$ is a cladode of a cactus graph $\Gamma$ with basepoint $p$, there is a unique vertex $b_C$ of $C$ closest to $p$, which we call the \emph{base} of $C$. Note that every vertex is the base of at least one cladode. We say that a cladode $C$ is \emph{above} a cladode $D$ if every path from $b_C$ to $p$ must pass through an edge of $D$. If $C$ is above $D$ there is a unique vertex $\proj_D(C)$ of $D$ closest to $b_C$, which we will call the \emph{projection} of $C$ onto $D$. Given two distinct points $x$ and $y$ in a common cladode $C$, with $x,y\ne b_C$, there is a unique reduced path from $x$ to $y$ in $C\setminus\{b_C\}$; if this path follows the orientation of $C$ we say $x$ is \emph{before} $y$, and otherwise we say $x$ is \emph{after} $y$. Within $C$ it also makes sense to say that an edge is before or after another edge, or that an edge is before or after a point not in the interior of that edge. We say a point or edge is \emph{between} two points or edges if it is before one and after the other.
\end{definition}

See Figure~\ref{fig:cactus} for an example illustrating the many definitions in Definition~\ref{def:cactus}.

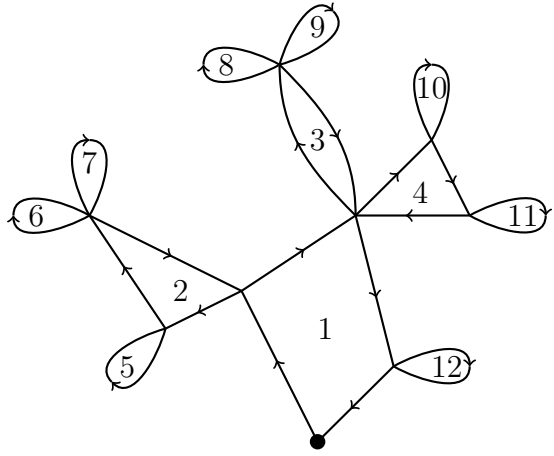
\begin{figure}[htb]
 \begin{tikzpicture}[line width=0.8pt]\centering
  \draw[-<-] (0,0) -- (1,1);\draw[-<-] (1,1) -- (0.5,3);\draw[-<-] (0.5,3) -- (-1,2);\draw[-<-] (-1,2) -- (0,0);
	\draw[-<-] (-1,2) -- (-3,3);\draw[-<-] (-3,3) -- (-2,1.5);\draw[-<-] (-2,1.5) -- (-1,2);
	\draw[-<-] (0.5,3) to [out=90, in=-45] (-0.5,5);\draw[-<-] (-0.5,5) to [out=-90, in=135] (0.5,3);
	\draw[-<-] (0.5,3) -- (2,3);\draw[-<-] (2,3) -- (1.5,4);\draw[-<-] (1.5,4) -- (0.5,3);
	\draw (1,1) to [out=-30, in=-90] (2,1);\draw[<-] (2,1) to [out=90, in=30] (1,1);
	\draw (2,3) to [out=-30, in=-90] (3,3);\draw[<-] (3,3) to [out=90, in=30] (2,3);
	\draw (1.5,4) to [out=60, in=0] (1.5,5);\draw[<-] (1.5,5) to [out=180, in=120] (1.5,4);
	\draw (-3,3) to [out=60, in=0] (-3,4);\draw[<-] (-3,4) to [out=180, in=120] (-3,3);
	\draw (-3,3) to [out=150, in=90] (-4,3);\draw[<-] (-4,3) to [out=-90, in=-150] (-3,3);
	\draw (-2,1.5) to [out=195, in=135] (-2-.707,1.5-.707);\draw[<-] (-2-.707,1.5-.707) to [out=-45, in=-105] (-2,1.5);
	\draw (-0.5,5) to [out=15, in=-45] (-0.5+.707,5+.707);\draw[<-] (-0.5+.707,5+.707) to [out=135, in=75] (-0.5,5);
	\draw (-0.5,5) to [out=150, in=90] (-1.5,5);\draw[<-] (-1.5,5) to [out=-90, in=-150] (-0.5,5);
	\filldraw (0,0) circle (2.5pt);
	\node at (0.1,1.5) {$1$}; \node at (-1.8,2) {$2$}; \node at (0,4) {$3$}; \node at (1.35,3.3) {$4$}; \node at (-2-.5,1.5-.5) {$5$}; \node at (-3-.7,3) {$6$}; \node at (-3,3+.7) {$7$}; \node at (-0.5-.7,5) {$8$}; \node at (-0.5+.5,5+.5) {$9$}; \node at (1.5,4+.7) {$10$}; \node at (2+.7,3) {$11$}; \node at (1+.7,1) {$12$};
 \end{tikzpicture}
 \caption{A cactus graph, with its cladodes numbered for reference. To illustrate the definitions in Definition~\ref{def:cactus} with some examples, we note: cladodes $3$ and $4$ have the same base; cladode $7$ is above cladodes $2$ and $1$ but no others; the projection of cladode $8$ onto cladode $1$ is the vertex that is the base of cladode $3$; and the base of cladode $3$ is after the base of cladode $2$ and before the base of cladode $12$, and hence is between them.}\label{fig:cactus}
\end{figure}

Given a graph $\Gamma$ and a subforest $F$, we will write $\Gamma/F$ for the graph obtained by quotienting each connected component of $F$ to a point. The quotient map $d\colon \Gamma \to \Gamma/F$ is called a \emph{forest collapse} or \emph{forest blow-down}. It is a homotopy equivalence, and a homotopy inverse of a forest blow-down is called a \emph{forest blow-up}, denoted $u\colon \Gamma/F \to \Gamma$.

\begin{definition}[Marking]\label{def:marking}
 A \emph{marking} of a basepointed graph $\Gamma$ is a homotopy equivalence $\rho \colon R_n \to \Gamma$ from the $n$-petaled rose to $\Gamma$, taking basepoint to basepoint.
\end{definition}

A marking of the rose itself represents an automorphism of $F_n$, thanks to our identification of $\Aut(F_n)$ with the group of basepoint-preserving self-homotopy equivalences of $R_n$, modulo homotopy. If a marking $\alpha \colon R_n \to R_n$ even represents a (pure) symmetric automorphism, then it makes sense to call the marking itself \emph{(pure) symmetric}. More generally:

\begin{definition}[(Pure) symmetric marking]\label{def:symm_marking}
 A marking $\rho \colon R_n \to \Gamma$ is called \emph{(pure) symmetric} if there exists a forest collapse $d\colon \Gamma \to R_n$ such that $d\circ \rho \colon R_n \to R_n$ is (pure) symmetric.
\end{definition}

\begin{definition}[Symmetric marked cactus graph]
 A \emph{symmetric marked cactus graph} is a triple $(\Gamma,p,\rho)$ where $\Gamma$ is a cactus graph with basepoint $p$ and $\rho$ is a symmetric marking. Two such triples $(\Gamma,p,\rho)$, $(\Gamma',p',\rho')$ are considered equivalent if there is a homeomorphism $\phi\colon \Gamma \to \Gamma'$ taking $p$ to $p'$ such that $\phi\circ \rho \simeq \rho'$. We will denote equivalence classes by $[\Gamma,p,\rho]$, and will usually just refer to $[\Gamma,p,\rho]$ as a symmetric marked cactus graph.
\end{definition}

We note that under this equivalence relation, every symmetric marked cactus graph is equivalent to one where the marking is even pure symmetric. This is just because the markings of the rose that permute the petals are all equivalent to the trivial marking. Moreover, these are the \emph{only} markings equivalent to the trivial marking, so the map $\alpha \mapsto [R_n,*,\alpha]$ is in fact a bijection between $\pgroup_n$ and the set of symmetric marked roses.

\begin{definition}[Partial order]
 We define a partial order $\le$ on the set of symmetric marked cactus graphs as follows. Let $[\Gamma,p,\rho]$ be a symmetric marked cactus graph and $F$ a subforest of $\Gamma$, with $d\colon \Gamma \to \Gamma/F$ the forest collapse. Let $\overline{p}_F \defeq d(p)$ and let $\overline{\rho}_F \defeq d \circ \rho$. We declare that $[\Gamma,p,\rho] \le [\Gamma/F,\overline{p}_F,\overline{\rho}_F]$. It is easy to check that the relation $\le$ is well defined up to equivalence of triples, and that it is a partial order.
\end{definition}

\begin{definition}[Complex of symmetric marked cactus graphs]
 The \emph{complex of symmetric marked cactus graphs} $\cpx_n$ is the geometric realization of the partially ordered set of symmetric marked cactus graphs.
\end{definition}

Note that $\group_n$ and $\pgroup_n$ act (on the right) on $\cpx_n^{(0)}$ via $[\Gamma,p,\rho].\alpha \defeq [\Gamma,p,\rho \circ \alpha]$, and this extends to an action on $\cpx_n$ since for any forest collapse $d\colon \Gamma \to \Gamma/F$, we have $(d\circ \rho) \circ \alpha = d\circ (\rho \circ \alpha)$, i.e., $\overline{\rho}_F \circ \alpha = \overline{(\rho \circ \alpha)}_F$.

\begin{cit}\cite[Proposition~3.5, Theorem~4.7]{collins89}
 The complex $\cpx_n$ is contractible and $(n-1)$-dimensional, and the actions of $\group_n$ and $\pgroup_n$ on $\cpx_n$ are proper and cocompact.
\end{cit}

Technically Collins considers the ``Outer'' version where we do not keep track of basepoints, but it is straightforward to get these results also in our basepointed ``Auter'' version.

In particular, we have the requisite setup of Definition~\ref{def:sig_invs}.

\begin{remark}
 One can similarly consider the complex of \emph{all} marked basepointed graphs, and get the well studied spine of \emph{Auter space}, which is contractible and on which $\Aut(F_n)$ acts properly and cocompactly (see \cite{culler86,hatcher98}). This is not relevant for our present purposes though, since the abelianization of $\Aut(F_n)$ is finite, and hence its character sphere is empty.
\end{remark}

\medskip

The next step is to take a character $\chi$ of $\pgroup_n$ and induce a character height function $h_\chi$ on $\cpx_n$. First recall that we equivocate between symmetric markings of roses and elements of $\pgroup_n$. Hence for $0$-simplices in $\cpx_n$ of the form $[R_n,*,\alpha]$, we can just define $h_\chi([R_n,*,\alpha])\defeq \chi(\alpha)$. In general we define $h_\chi([\Gamma,p,\rho])$ as follows:

\begin{definition}[The character height function $h_\chi$]
 Let $[\Gamma,p,\rho]$ be a symmetric marked cactus graph. Define $h_\chi([\Gamma,p,\rho])$ to be
 $$h_\chi([\Gamma,p,\rho]) \defeq \max\{\chi(\alpha)\mid [\Gamma,p,\rho]\in \st([R_n,*,\alpha])\}\text{.}$$
 Extend this affinely to the simplices of $\cpx_n$, to get $h_\chi \colon \cpx_n \to \R$.
\end{definition}

\begin{observation}
 $h_\chi$ is a character height function.
\end{observation}

\begin{proof}
 We need to show that $h_\chi([\Gamma,p,\rho].\alpha) = h_\chi([\Gamma,p,\rho]) + \chi(\alpha)$ for all $[\Gamma,p,\rho]\in \cpx_n^{(0)}$ and $\alpha\in \pgroup_n$. We know that $[\Gamma,p,\rho].\alpha = [\Gamma,p,\rho\circ \alpha]$, and clearly $[\Gamma,p,\rho]\in \st [R_n,*,\beta]$ if and only if $[\Gamma,p,\rho\circ \alpha]\in \st [R_n,*,\beta\circ \alpha]$, so this follows simply because $\chi(\beta\circ \alpha)=\chi(\beta)+\chi(\alpha)$ for all $\alpha,\beta\in \pgroup_n$.
\end{proof}

Now we need a ``tiebreaker'' function $f$ as in Lemma~\ref{lem:bnsr_morse}. As discussed before and after that lemma, any randomly chosen injective $\overline{f}\colon \cpx_n^{(0)}/G \to \R$ could serve to induce a tiebreaker $f\colon \cpx_n \to \R$, but we want to be more clever than this. In particular our tiebreaker will yield tractable ascending links that will actually reveal parts of the BNSR-invariants.

The $0$-cells in the orbit space $\cpx_n/\pgroup_n$ are homeomorphism classes of cactus graphs, so ``number of vertices'' is a well defined measurement on these $0$-cells. Let $\overline{f}\colon \cpx_n^{(0)}/G \to \R$ be the function taking a graph to the \emph{negative} of its number of vertices. In particular since we are using the negative, the rose has the largest $\overline{f}$ value of all cactus graphs. Let $f\colon \cpx_n \to \R$ be the extension of $\overline{f}$ described before Lemma~\ref{lem:bnsr_morse}, so $f([\Gamma,p,\rho])$ equals the negative number of vertices of $\Gamma$, and consider the function $(h_\chi,f)\colon \cpx_n \to \R\times \R$. Since $\cpx_n$ is simplicial and adjacent $0$-simplices in $\cpx_n$ cannot share a $\pgroup_n$-orbit (for instance since they necessarily have different $f$ values), the following is immediate from Lemma~\ref{lem:bnsr_morse}:

\begin{corollary}
 For any $\chi$, $(h_\chi,f)$ is a Morse function on $\cpx_n$. \qed
\end{corollary}

It is clear from the definition of $h_\chi$ that $\cpx_n^{h_\chi\ge t}$ is the union of the stars of those $[R_n,*,\alpha]$ with $\chi(\alpha)\ge t$. It is a common phenomenon when working in Auter space and its relatives to encounter important subcomplexes that are unions of stars of marked roses. Another example arises in \cite{bestvina07}, where Bestvina, Bux and Margalit use ``homology markings'' of roses to prove that for $n\ge 3$ the kernel of $\Out(F_n)\to \GL_n(\Z)$ has cohomological dimension $2n-4$ and is not of type $\F_{2n-4}$ (when $n\ge 4$ it remains open whether or not this kernel is of type $\F_{2n-5}$).

\medskip

We record here a useful technical lemma that gives information on how $h_\chi$ can differ between ``nearby'' symmetric marked roses. Let $[\Gamma,p,\rho]$ be a symmetric marked cactus graph and let $T$ be a spanning tree in $\Gamma$. Since $T$ is spanning, collapsing $T$ yields a symmetric marked rose. The marking $\rho$ provides the cladodes of $\Gamma$ with a numbering from $1$ to $n$; let $C_{i,\rho}$ be the $i$th cladode. Since $T$ is a spanning tree, it meets $C_{i,\rho}$ at all but one edge; write $E_{i,T}$ for the single-edge subforest of $C_{i,\rho}$ that is not in $T$. In particular, intuitively, upon collapsing $T$, $E_{i,T}$ becomes the $i$th petal of $R_n$. Note that $T$ is completely determined by the set $\{E_{1,T},\dots,E_{n,T}\}$, namely it consists of all the edges of $\Gamma$ not in any $E_{i,T}$.

\begin{lemma}[Change of spanning tree]\label{lem:change_tree}
 Let $[\Gamma,p,\rho]$ be a symmetric marked cactus graph and let $T$ be a spanning tree in $\Gamma$. Suppose $U$ is another spanning tree such that $E_{j,T}\ne E_{j,U}$ but $E_{i,T}=E_{i,U}$ for all $i\ne j$ (so $U$ differs from $T$ only in the $j$th cladode). Suppose that $E_{j,T}$ is before $E_{j,U}$ (in the language of Definition~\ref{def:cactus}). Let $\emptyset\ne I\subsetneq [n]$ be the set of indices $i$ such that the projection $\proj_{C_{j,\rho}}(C_{i,\rho})$ lies between $E_{j,T}$ and $E_{j,U}$ (so in particular $j\not\in I$). Then for any $\chi=\sum_{i\ne j} a_{i,j} \chi_{i,j}$ we have $h_\chi([\Gamma/T,\overline{p}_T,\overline{\rho}_T]) - h_\chi([\Gamma/U,\overline{p}_U,\overline{\rho}_U]) = \sum_{i\in I}a_{i,j}$.
\end{lemma}

\begin{proof}
 By collapsing the subforest $T\cap U$ we can assume without loss of generality that $T=E_{j,T}$ and $U=E_{j,U}$ are each a single edge, so $\Gamma$ has two vertices, the basepoint $p$ and another vertex $q$. The set $I$ indexes those cladodes whose base is $q$, so $[n]\setminus I$ indexes those cladodes whose base is $p$. Up to the action of $\pgroup_n$ we can assume that $\Gamma/U$ is the trivially marked rose, so we need to show that $h_\chi([\Gamma/T,\overline{p}_T,\overline{\rho}_T]) = \sum_{i\in I}a_{i,j}$. In fact the procedure of blowing up the trivial rose to get $\Gamma$ and then blowing down $T$ is a Whitehead move (see \cite[Section~3.1]{culler86}) that corresponds to the symmetric Whitehead automorphism $\alpha_{I,j}$. In other words, viewed as an element of $\pgroup_n$, we have $\overline{\rho}_T = \alpha_{I,j}$. This means that $h_\chi([\Gamma/T,\overline{p}_T,\overline{\rho}_T]) = \chi(\alpha_{I,j}) = \sum_{i\in I}a_{i,j}$, as desired.
\end{proof}

\medskip

\textbf{Ascending links:} In the next section we will need to understand ascending links $\alk v$ with respect to $(h_\chi,f)$, for $v=[\Gamma,p,\rho]$ a $0$-simplex in $\cpx_n$, so we discuss this a bit here. Since $\alk v$ is a full subcomplex of $\lk v$ we just need to understand which $0$-simplices of $\lk v$ lie in $\alk v$. First note that $\lk v$ is a join, of the \emph{down-link} $\dlk v$, spanned by those $0$-simplices of $\lk v$ obtained from forest blow-downs of $\Gamma$, and its \emph{up-link} $\ulk v$, spanned by those $0$-simplices of $\lk v$ corresponding to forest blow-ups of $\Gamma$. The ascending link $\alk v$ similarly decomposes as the join of the \emph{ascending down-link} $\adlk v$ and \emph{ascending up-link} $\aulk v$, which are just defined to be $\adlk v \defeq \dlk v \cap \alk v$ and $\aulk v \defeq \ulk v \cap \alk v$.

Since $0$-simplices in $\dlk v$ have larger $f$ value than $v$ (i.e., the graphs have fewer vertices than $\Gamma$) and cannot have strictly larger $h_\chi$ value, given a subforest $F\subseteq \Gamma$ we see that $[\Gamma/F,\overline{p}_F,\overline{\rho}_F]\in \adlk v$ if and only if $h_\chi([\Gamma/F,\overline{p}_F,\overline{\rho}_F])\ge h_\chi([\Gamma,p,\rho])$ if and only if $h_\chi([\Gamma/F,\overline{p}_F,\overline{\rho}_F])= h_\chi([\Gamma,p,\rho])$. Similarly if $[\widetilde{\Gamma},\widetilde{p},\widetilde{\rho}]\in \ulk v$ then it lies in $\aulk v$ if and only if it has strictly larger $h_\chi$ value than $v$ (since it has smaller $f$ value).


\section{Topology of ascending links}\label{sec:asc_link}

Throughout this section, $\chi$ is a non-trivial character of $\pgroup_n$ with character height function $h_\chi$ on $\cpx_n$, and $\alk [\Gamma,p,\rho]$ means the ascending link of the $0$-simplex $[\Gamma,p,\rho]$ with respect to $\chi$.

We will analyze the topology of $\alk [\Gamma,p,\rho] = \adlk [\Gamma,p,\rho] * \aulk [\Gamma,p,\rho]$ by inspecting $\adlk [\Gamma,p,\rho]$ and $\aulk [\Gamma,p,\rho]$ individually. First we focus on $\adlk [\Gamma,p,\rho]$.

\subsection{Ascending down-link}\label{sec:asc_down_link}

The first goal is to realize $\adlk [\Gamma,p,\rho]$ as a nice combinatorial object, the \emph{complex of ascending forests}.

\begin{definition}[Complex of forests]
 The \emph{complex of forests} $\forests(\Gamma)$ for a graph $\Gamma$ is the geometric realization of the partially ordered set of non-empty subforests of $\Gamma$, with partial order given by inclusion.
\end{definition}

We will not really need to know the homotopy type of $\forests(\Gamma)$ in what follows, but it is easy to compute so we record it here for good measure.

\begin{observation}\label{obs:forest_cpx}
 For $\Gamma$ a cactus graph with $V$ vertices, $\forests(\Gamma)\simeq S^{V-2}$.
\end{observation}

\begin{proof}
 For each cladode $C$ let $\forests_C(\Gamma)$ be the complex of subforests of $\Gamma$ contained in $C$. Clearly $\forests_C(\Gamma) \simeq S^{V_C-2}$, where $V_C$ is the number of vertices of $C$. Since $\forests(\Gamma)$ is the join of all the $\forests_C(\Gamma)$, we get $\forests_C(\Gamma)\simeq S^{d-1}$ for $d=\sum_C (V_C-1)$. Now, $V_C-1$ is the number of non-base vertices of $C$, and every vertex of $\Gamma$ except for the basepoint is a non-base vertex of a unique cladode, so $\sum_C (V_C-1) = V-1$.
\end{proof}

\begin{definition}[Complex of ascending forests]
 The \emph{complex of ascending forests} $\aforests(\Gamma,p,\rho)$ for a symmetric marked cactus graph $[\Gamma,p,\rho]$ is the full subcomplex of $\forests(\Gamma)$ supported on those $0$-simplices $F$ such that $[\Gamma/F,\overline{p}_F,\overline{\rho}_F]\in \alk [\Gamma,p,\rho]$.
\end{definition}

\begin{observation}
 $\aforests(\Gamma,p,\rho) \cong \adlk [\Gamma,p,\rho]$.
\end{observation}

\begin{proof}
 The isomorphism is given by $F\mapsto [\Gamma/F,\overline{p}_F,\overline{\rho}_F]$.
\end{proof}

\medskip

For certain characters, $\aforests(\Gamma,p,\rho)$ is guaranteed to be contractible. We call these characters \emph{decisive}:

\begin{definition}[Decisive]
 Call a character $\chi$ of $\pgroup_n$ \emph{decisive} if every $[\Gamma,p,\rho]$ lies in a unique star of a symmetric marked rose with maximal $\chi$ value.
\end{definition}

\begin{observation}
 Let $\chi$ be decisive. Then for any $\Gamma\ne R_n$, $\aforests(\Gamma,p,\rho)$ is contractible.
\end{observation}

\begin{proof}
 Every ascending forest is contained in an ascending spanning tree, and we are assuming there is a unique ascending spanning tree, so $\aforests(\Gamma,p,\rho)$ is just the star in $\aforests(\Gamma,p,\rho)$ of this unique ascending spanning tree.
\end{proof}

\begin{proposition}[Positive implies decisive]\label{prop:positive_decisive}
 Positive characters of $\pgroup_n$ are decisive.
\end{proposition}

\begin{proof}
 Let $T$ be the spanning tree in $\Gamma$ such that, using the notation from Lemma~\ref{lem:change_tree}, for each cladode $C_{j,\rho}$, the origin of the edge in $E_{j,T}$ is the base of $C_{j,\rho}$; see Figure~\ref{fig:positive_asc_spanning_tree} for an example. Note that the edge of $E_{j,T}$ is before all the other edges of $C_{j,\rho}$. We claim that $\overline{\rho}_T$ has larger $\chi$ value than $\overline{\rho}_U$ for any other spanning tree $U$. First we prove this in the case when $U$ differs from $T$ only in one cladode, say $C_{j,\rho}$. Let $I_j$ be the set of $i$ such that the projection $\proj_{C_{j,\rho}}(C_{i,\rho})$ lies between $E_{j,U}$ and $E_{j,T}$. Since $E_{j,T}$ is before $E_{j,U}$, by Lemma~\ref{lem:change_tree} we get $h_\chi([\Gamma/T,\overline{p}_T,\overline{\rho}_T]) - h_\chi([\Gamma/U,\overline{p}_U,\overline{\rho}_U]) = \sum_{i\in I_j}a_{i,j}>0$. Now suppose $U$ differs from $T$ in more than one cladode, say $C_{j_1,\rho},\dots,C_{j_r,\rho}$. By changing $E_{j_k,T}$ to $E_{j_k,U}$ one $k$ at a time, we get $h_\chi([\Gamma/T,\overline{p}_T,\overline{\rho}_T]) - h_\chi([\Gamma/U,\overline{p}_U,\overline{\rho}_U]) = \sum_{k=1}^r\sum_{i\in I_{j_k}}a_{i,j_k}>0$. We conclude that $h_\chi([\Gamma/T,\overline{p}_T,\overline{\rho}_T]) > h_\chi([\Gamma/U,\overline{p}_U,\overline{\rho}_U])$, as desired.
\end{proof}

By a parallel argument, negative characters are also decisive.

\begin{figure}[htb]
 \begin{tikzpicture}[line width=2.5pt]\centering
  \draw[-<-] (0,0) -- (1,1);\draw[-<-] (1,1) -- (0.5,3);\draw[-<-] (0.5,3) -- (-1,2);\draw[line width=0.8pt,-<-] (-1,2) -- (0,0);
	\draw[-<-] (-1,2) -- (-3,3);\draw[-<-] (-3,3) -- (-2,1.5);\draw[line width=0.8pt,-<-] (-2,1.5) -- (-1,2);
	\draw[-<-] (0.5,3) to [out=90, in=-45] (-0.5,5);\draw[line width=0.8pt,-<-] (-0.5,5) to [out=-90, in=135] (0.5,3);
	\draw[-<-] (0.5,3) -- (2,3);\draw[-<-] (2,3) -- (1.5,4);\draw[line width=0.8pt,-<-] (1.5,4) -- (0.5,3);
	\draw[line width=0.8pt] (1,1) to [out=-30, in=-90] (2,1);\draw[line width=0.8pt,<-] (2,1) to [out=90, in=30] (1,1);
	\draw[line width=0.8pt] (2,3) to [out=-30, in=-90] (3,3);\draw[line width=0.8pt,<-] (3,3) to [out=90, in=30] (2,3);
	\draw[line width=0.8pt] (1.5,4) to [out=60, in=0] (1.5,5);\draw[line width=0.8pt,<-] (1.5,5) to [out=180, in=120] (1.5,4);
	\draw[line width=0.8pt] (-3,3) to [out=60, in=0] (-3,4);\draw[line width=0.8pt,<-] (-3,4) to [out=180, in=120] (-3,3);
	\draw[line width=0.8pt] (-3,3) to [out=150, in=90] (-4,3);\draw[line width=0.8pt,<-] (-4,3) to [out=-90, in=-150] (-3,3);
	\draw[line width=0.8pt] (-2,1.5) to [out=195, in=135] (-2-.707,1.5-.707);\draw[line width=0.8pt,<-] (-2-.707,1.5-.707) to [out=-45, in=-105] (-2,1.5);
	\draw[line width=0.8pt] (-0.5,5) to [out=15, in=-45] (-0.5+.707,5+.707);\draw[line width=0.8pt,<-] (-0.5+.707,5+.707) to [out=135, in=75] (-0.5,5);
	\draw[line width=0.8pt] (-0.5,5) to [out=150, in=90] (-1.5,5);\draw[line width=0.8pt,<-] (-1.5,5) to [out=-90, in=-150] (-0.5,5);
	\filldraw (0,0) circle (2.5pt);
 \end{tikzpicture}
 \caption{A cactus graph, with the tree $T$ from the proof of Proposition~\ref{prop:positive_decisive} marked in bold.}\label{fig:positive_asc_spanning_tree}
\end{figure}
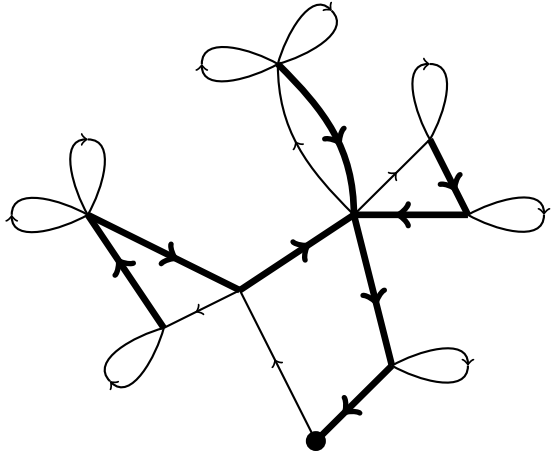

It turns out ``most'' characters are decisive, in the following sense:

\begin{definition}[Generic]\label{def:generic}
 Call a character $\chi=\sum_{i\ne j} a_{i,j}\chi_{i,j}$ of $\pgroup_n$ \emph{generic} if for every choice of $\varepsilon_{i,j}\in\{-1,0,1\}$ we have $\sum_{i\ne j} \varepsilon_{i,j} a_{i,j} = 0$ only if $\varepsilon_{i,j}=0$ for all $i,j$ (said another way, the $a_{i,j}$ have no non-trivial linear dependencies using coefficients from $\{-1,0,1\}$).
\end{definition}

\begin{observation}\label{obs:generic_dense}
 The set $\{[\chi]\in S(\pgroup_n) \mid \chi$ is generic$\}$ is dense in $S(\pgroup_n)$.
\end{observation}

\begin{proof}
 Given a linear dependence $\sum_{i\ne j} \varepsilon_{i,j} a_{i,j}=0$ with $\varepsilon_{i,j}\in\{-1,0,1\}$, the complement of the set of character classes satisfying this dependence is open and dense in $S(\pgroup_n)$. Since there are only finitely many choices for the $\varepsilon_{i,j}$, the set of generic character classes is also (open and) dense in $S(\pgroup_n)$.
\end{proof}

\begin{proposition}[Generic implies decisive]\label{prop:generic_decisive}
 Generic characters of $\pgroup_n$ are decisive.
\end{proposition}

\begin{proof}
 Let $T$ and $U$ be two different spanning trees in $\Gamma$, so $[\Gamma,p,\rho]$ lies in the stars of the symmetric marked roses $[\Gamma/T,\overline{p}_T,\overline{\rho}_T]$ and $[\Gamma/U,\overline{p}_U,\overline{\rho}_U]$. We claim that for $\chi$ generic, these symmetric marked roses have different $h_\chi$ values, from which the result will follow. Using the notation from Lemma~\ref{lem:change_tree}, suppose the cladodes in which $T$ and $U$ differ are $C_{j_1,\rho},\dots,C_{j_r,\rho}$, so $E_{j_k,T}\ne E_{j_k,U}$ for $1\le k\le r$, but $E_{j,T}=E_{j,U}$ for all $j\not\in \{j_1,\dots,j_r\}$. Let $U=T_0,T_1,\dots,T_r=T$ be spanning trees such that we obtain $T_k$ from $T_{k-1}$ by replacing $E_{j_k,U}$ with $E_{j_k,T}$. For each $1\le k\le r$ let $I_k$ be the set of $i$ such that the projection $\proj_{C_{j_k,\rho}}(C_{i,\rho})$ lies between $E_{j_k,T}$ and $E_{j_k,U}$. By Lemma~\ref{lem:change_tree}, we know that
$$h_\chi([\Gamma/T_k,\overline{p}_{T_k},\overline{\rho}_{T_k}]) - h_\chi([\Gamma/T_{k-1},\overline{p}_{T_{k-1}},\overline{\rho}_{T_{k-1}}]) = \pm\sum_{i\in I_k}a_{i,j_k}$$
for each $1\le k\le r$ (with the plus or minus depending on whether $E_{j_k,T}$ is before or after $E_{j_k,U}$). This implies that $h_\chi([\Gamma/T,\overline{p}_T,\overline{\rho}_T]) - h_\chi([\Gamma/U,\overline{p}_U,\overline{\rho}_U]) = (\pm \sum_{i\in I_1}a_{i,j_1}) + (\pm \sum_{i\in I_2}a_{i,j_2}) + \cdots + (\pm \sum_{i\in I_r}a_{i,j_r})$. Since $\chi$ is generic, this cannot be zero.
\end{proof}

\begin{remark}
 If $\chi$ is not decisive, then $\aforests(\Gamma,p,\rho)$ is still somewhat understandable, it just might not be contractible. A forest is ascending if and only if it lies in an ascending spanning tree, so $\aforests(\Gamma,p,\rho)$ is the union of the stars of the ascending spanning trees. Also, a non-empty intersection of some of these stars is again a (contractible) star, namely the star of the forest that is the intersection of the relevant spanning trees. Hence $\aforests(\Gamma,p,\rho)$ is homotopy equivalent to the nerve of its covering by stars of ascending spanning trees. This is isomorphic to the simplicial complex whose $0$-simplices are the ascending spanning trees, and where $k$ of them span a $(k-1)$-simplex whenever the trees have a non-empty intersection. In theory it should be possible to compute the homotopy type of this complex, but we are not currently interested in the non-decisive characters (since they will be totally intractable when we study the ascending up-link in the next subsection), so we leave further analysis of this for the future.
\end{remark}

Since $\alk v =\adlk v * \aulk v$ and $\adlk [\Gamma,p,\rho] \cong \aforests(\Gamma,p,\rho)$, we get:

\begin{corollary}\label{cor:alk_non_rose_cible}
 If $\chi$ is a decisive character of $\pgroup_n$ and $v=[\Gamma,p,\rho]$ for $\Gamma$ not a rose, then $\alk v$ is contractible. \qed
\end{corollary}

\subsection{Ascending up-link}\label{sec:asc_up_link}

Thanks to Corollary~\ref{cor:alk_non_rose_cible}, for decisive characters of $\pgroup_n$ the only $0$-simplices of $\cpx_n$ that can have non-contractible ascending links are those of the form $[R_n,*,\alpha]$. These have empty down-link, so the ascending link equals the ascending up-link. It turns out that the ascending up-link $\aulk [R_n,*,\alpha]$ is homotopy equivalent to a particularly nice complex $\aie_n(\chi)$, the complex of \emph{ascending ideal edges}, which we now describe.

Let $E(*)$ be the set of half-edges of $R_n$ incident to $*$. Since we identify $\pi_1(R_n)$ with $F_n$, the petals of $R_n$ are naturally identified with the basis $S=\{x_1,\dots,x_n\}$ of $F_n$. We will write $i$ for the half-edge in the petal $x_i$ with $*$ as its origin, and $\overline{i}$ for the half-edge in $x_i$ with $*$ as its terminus, so $E(*)=\{1,\overline{1},2,\overline{2},\dots,n,\overline{n}\}$.

\begin{definition}[(Symmetric) ideal edges]
 A subset $A$ of $E(*)$ such that $|A|\ge 2$ and $|E(*)\setminus A|\ge 1$ is called an \emph{ideal edge}. We say an ideal edge $A$ \emph{splits} $\{i,\overline{i}\}$ if $\{i,\overline{i}\}\cap A$ and $\{i,\overline{i}\}\setminus A$ are both non-empty. We call $A$ \emph{symmetric} if there exists precisely one $i\in [n]$ such that $A$ splits $\{i,\overline{i}\}$. In this case we call $\{i,\overline{i}\}$ the \emph{split pair} of $A$.
\end{definition}

Intuitively, an ideal edge $A$ describes a way of blowing up a new edge at $*$, with the half-edges in $E(*)\setminus A$ becoming incident to the new basepoint and the half-edges in $A$ becoming incident to the new non-basepoint vertex; a more rigorous discussion can be found for example in \cite{jensen02}. The conditions in the definition ensure that blowing up a symmetric ideal edge results in a cactus graph. See Figure~\ref{fig:symm_ideal_edges} for an example. The asymmetry between the conditions $|A|\ge 2$ and $|E(*)\setminus A|\ge 1$ arises because the basepoint of a cactus graph must have degree at least $2$, whereas other vertices must have degree at least $3$. In fact every vertex of a cactus graph has even degree, so in practice $|A|\ge 2$ is equivalent to $|A|\ge 3$ for symmetric ideal edges.

\begin{figure}[htb]
 \begin{tikzpicture}[line width=0.8pt]\centering

  \begin{scope}
	 \filldraw (0,0) circle (1.5pt) (1,0) circle (1.5pt)   (3,0) circle (1.5pt) (4,0) circle (1.5pt);
   \draw[line width=1.5pt] (2.5,0) ellipse (2 cm and 1.25 cm);
	 \node at (0,0.35) {$1$};\node at (1,0.35) {$\overline{1}$};\node at (3,0.35) {$2$};\node at (4,0.35) {$\overline{2}$};
	 \node at (2.5,-0.5) {$A$};\node at (-0.5,-0.5) {$E(*)\setminus A$};
	 \node at (6,0) {$\longrightarrow$};
	\end{scope}
	
  \begin{scope}[yshift=-4cm]
	 \filldraw (0,0) circle (1.5pt) (1,0) circle (1.5pt)   (3,0) circle (1.5pt) (4,0) circle (1.5pt);
   \draw[line width=1.5pt] (2,0) ellipse (1.5 cm and 1.25 cm);
	 \node at (0,0.35) {$1$};\node at (1,0.35) {$\overline{1}$};\node at (3,0.35) {$2$};\node at (4,0.35) {$\overline{2}$};
	 \node at (2,-0.5) {$A$};\node at (-0.5,-0.5) {$E(*)\setminus A$};
	 \node at (6,0) {$\longrightarrow$};
	\end{scope}
	
	\begin{scope}[xshift=8cm,yshift=-1cm]
	 \draw[->-] (0,0) to [out=120,in=-120] (0,2);\draw[line width=2.5pt] (0,2) to [out=-60,in=60] (0,0);
	 \draw[->-] (0,2) to [out=120,in=180] (0,3);\draw[->-] (0,3) to [out=0,in=60,->-] (0,2);
	 \node at (-0.35,0.25) {$1$};\node at (-0.35,1.75) {$\overline{1}$};\node at (-0.35,2.25) {$2$};\node at (0.35,2.25) {$\overline{2}$};
	 \filldraw (0,0) circle (2.5pt);
	\end{scope}
	
	\begin{scope}[xshift=8cm,yshift=-5cm]
	 \draw[->-] (0,0) to [out=180,in=180] (0,2);\draw[->-] (0,2) to [out=0,in=0] (0,0);\draw[line width=2.25pt] (0,2) -- (0,0);
	 \node at (-0.6,0) {$1$};\node at (-0.6,2) {$\overline{1}$};\node at (0.6,0) {$\overline{2}$};\node at (0.6,2) {$2$};
	 \filldraw (0,0) circle (2.5pt);
	\end{scope}
	
 \end{tikzpicture}
 \caption{The symmetric ideal edge $\{\overline{1},2,\overline{2}\}$ and the non-symmetric ideal edge $\{\overline{1},2\}$, together with the blow-ups they produce. The former yields a cactus graph and the latter does not.}\label{fig:symm_ideal_edges}
\end{figure}
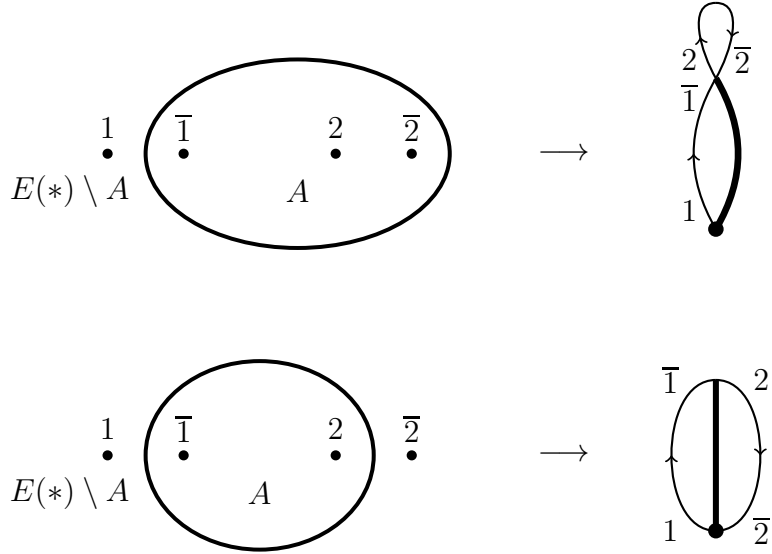

\begin{definition}[Ascending symmetric ideal edge]\label{def:asc_symm_ideal_edge}
 Let $\chi=\sum_{i\ne j} a_{i,j}\chi_{i,j}$ be a character of $\pgroup_n$ and let $A$ be a symmetric ideal edge. Suppose $\{j,\overline{j}\}$ is the split pair of $A$ and let $I=(A\cap [n])\setminus\{j\}$. We call $A$ \emph{ascending} (with respect to $\chi$) if either
 \begin{enumerate}
  \item $j\in A$ and $\sum_{i\in I} a_{i,j} > 0$ or
	\item $\overline{j}\in A$ and $\sum_{i\in I} a_{i,j} < 0$.
 \end{enumerate}
\end{definition}

For example, the symmetric ideal edge in Figure~\ref{fig:symm_ideal_edges} is ascending if and only if $a_{2,1}<0$. If $\chi$ is positive (respectively negative) then $A$ is ascending if and only if $j\in A$ (respectively $\overline{j}\in A$), for $\{j,\overline{j}\}$ the split pair of $A$. If $\chi$ is generic then for any set $I$ of pairs $\{i,\overline{i}\}$ and any $j\in [n]\setminus I$, $\sum_{i\in I} a_{i,j} \ne 0$, so one of $A=\{j\}\cup I$ or $A=\{\overline{j}\}\cup I$ is an ascending ideal edge.

\begin{definition}[Compatible]
 Two ideal edges $A$ and $A'$ are called \emph{compatible} if any of $A\subseteq A'$, $A'\subseteq A$ or $A\cap A' = \emptyset$ occur.
\end{definition}

\begin{definition}[Complex of (ascending) symmetric ideal edges]
 Let $\ie_n$ be the simplicial complex whose $0$-simplices are the symmetric ideal edges, and where a collection of symmetric ideal edges span a simplex if and only if they are pairwise compatible. Let $\aie_n(\chi)$ be the subcomplex of $\ie_n$ spanned by the ascending symmetric ideal edges.
\end{definition}

It is a classical fact that $\ie_n$ is homotopy equivalent to $\ulk [R_n,*,\alpha]$ (for any $\alpha$). More precisely, the barycentric subdivision $\ie_n'$ is isomorphic to $\ulk [R_n,*,\alpha]$. It turns out a similar thing happens when restricting to ascending ideal edges:

\begin{lemma}
 For any character $\chi$ of $\pgroup_n$ and any $0$-simplex of the form $[R_n,*,\alpha]$, $\aulk [R_n,*,\alpha]$ is homotopy equivalent to $\aie_n(\chi)$.
\end{lemma}

\begin{proof}
 Since the barycentric subdivision $\ie_n'$ of $\ie_n$ is isomorphic to $\ulk [R_n,*,\alpha]$, we have that $\aulk [R_n,*,\alpha]$ is isomorphic to a subcomplex $\ie_n'(asc)$ of $\ie_n'$. This is the subcomplex spanned by those $0$-simplices in $\ie_n'$, i.e, those collections of pairwise compatible symmetric ideal edges, whose corresponding tree blow-up makes $h_\chi$ go up. Note that $\aie_n(\chi)'$ is a subcomplex of $\ie_n'(asc)$, since as soon as one ideal edge in a collection corresponds to an ascending edge blow-up the whole collection corresponds to an ascending tree blow-up.

Given a $0$-simplex $\sigma=\{A_1,\dots,A_k\}$ of $\ie_n'(asc)$, let $\phi(\sigma)\defeq \{A_i\mid A_i\in \aie_n(\chi)\}$. We claim that $\phi(\sigma)$ is non-empty, and hence $\phi \colon \ie_n'(asc) \to \ie_n'(asc)$ is a well defined map whose image is $\aie_n(\chi)$. Let $[\Gamma,p,\rho]$ be the result of blowing up the ideal tree given by $\sigma$. Let $U$ be the spanning tree in $\Gamma$ such that $[\Gamma/U,\overline{p}_U,\overline{\rho}_U] = [R_n,*,\alpha]$. Since the blow-up is ascending, the blow-down reversing it cannot be ascending, so $U$ is not an ascending spanning tree in $\Gamma$. Choose an ascending spanning tree $T$ in $\Gamma$, so $U\ne T$. Similar to the proof of Proposition~\ref{prop:generic_decisive}, we can turn $U$ into $T$ by changing one edge at a time, and from Lemma~\ref{lem:change_tree} we get
$$h_\chi([\Gamma/T,\overline{p}_T,\overline{\rho}_T]) - h_\chi([\Gamma/U,\overline{p}_U,\overline{\rho}_U]) = (\pm \sum_{i\in I_1}a_{i,j_1}) + (\pm \sum_{i\in I_2}a_{i,j_2}) + \cdots + (\pm \sum_{i\in I_r}a_{i,j_r})$$
where the $I_k$ and $j_k$ are as in the proof of Proposition~\ref{prop:generic_decisive}. Since $T$ is ascending but $U$ is not, this quantity is positive. Hence there exists $k$ such that $\pm \sum_{i\in I_k}a_{i,j_k} >0$ (with the ``$\pm$'' determined by whether $E_{j_k,T}$ is before or after $E_{j_k,U}$). Write $j=j_k$ for brevity.

Now let $T'$ be the spanning tree $(T\setminus E_{j,U})\cup E_{j,T}$ (keep in mind that $E_{j,U}$ lies in $T$ and not $U$, and $E_{j,T}$ lies in $U$ and not $T$), so, roughly, $T'$ is the result of changing only the part of $T$ in $C_{j,\rho}$ to look like $U$. Let $F\defeq T\setminus C_{j,\rho}$ and consider $[\Gamma/F,\overline{p}_F,\overline{\rho}_F]$. Let $\overline{E}_{j,U}$ and $\overline{E}_{j,T}$ be the images of $E_{j,U}$ and $E_{j,T}$ in $\Gamma/F$. The difference between the $h_\chi$ values obtained by blowing down $\overline{E}_{j,U}$ versus $\overline{E}_{j,T}$ is the positive value $\pm \sum_{i\in I_k}a_{i,j_k}$ from before; hence $\overline{E}_{j,U}$ is ascending in $\Gamma/F$ and $\overline{E}_{j,T}$ is not ascending. Now, the blow-up of $[R_n,*,\alpha]$ resulting in $[\Gamma/F,\overline{p}_F,\overline{\rho}_F]$ corresponds to one of the $A_i$, and $\overline{E}_{j,T}$ is the new edge blown up. This is an ascending blow-up, since the reverse is a non-ascending blow-down. This shows that at least one of the $A_i$ is indeed ascending, so $\phi(\sigma)\ne\emptyset$.

Having shown that $\phi \colon \ie_n'(asc) \to \ie_n'(asc)$ is well defined, it is easily seen to be a poset retraction (\`a la \cite[Section~1.3]{quillen78}) onto its image $\aie_n(\chi)'$, so we conclude that $\aulk [R_n,*,\alpha] \cong \ie_n'(asc) \simeq \aie_n(\chi)' \cong \aie_n(\chi)$.
\end{proof}

\medskip

It is clear from Definition~\ref{def:asc_symm_ideal_edge} that for $\chi$ positive, the complex $\aie_n(\chi)$ is independent of $\chi$. We will write $\aie_n(pos)$ for $\aie_n(\chi)$ in this case. In Proposition~\ref{prop:pos_alk_conn} we will determine the connectivity properties of $\aie_n(pos)$. First we need the following useful lemma, which was proved in \cite{witzel16}.

\begin{lemma}[Strong Nerve Lemma]\label{lem:nerve}
 Let $Y$ be a simplicial complex covered by subcomplexes $Y_1,\dots,Y_n$. Suppose that whenever an intersection $\bigcap_{i=1}^k Y_{j_i}$ of $k$ of them ($1\le k\le n$) is non-empty, it is $(n-k-2)$-connected. If the nerve $N$ of the covering is $(n-3)$-connected then so is $Y$. If the nerve of the covering is $(n-3)$-connected but not $(n-2)$-acyclic, then so is $Y$.
\end{lemma}

\begin{proof}
 That $Y$ is $(n-3)$-connected follows from the usual Nerve Lemma, e.g., \cite[Lemma~1.2]{bjoerner94}, but this usual Nerve Lemma is not enough to show $Y$ is not $(n-2)$-acyclic if the nerve is not. In \cite[Proposition~1.21]{witzel16} it was shown using spectral sequences that indeed these hypotheses ensure that $Y$ is not $(n-2)$-acyclic.
\end{proof}

\begin{proposition}\label{prop:pos_alk_conn}
 The complex $\aie_n(pos)$ is $(n-3)$-connected but not $(n-2)$-acyclic (and hence so are $\aulk [R_n,*,\alpha]$ and $\alk [R_n,*,\alpha]$ for any positive character of $\pgroup_n$).
\end{proposition}

\begin{proof}
 We will prove that $\aie_n(pos)$ is $(n-3)$-connected by using induction to prove a more general statement, and then afterwards we will prove that $\aie_n(pos)$ is not $(n-2)$-acyclic by applying Lemma~\ref{lem:nerve}. Call a subset $P\subseteq E(*)$ \emph{positive} if for each $1\le i\le n$ we have that $\overline{i}\in P$ implies $i\in P$. Define the \emph{defect} $d(P)$ to be the number of $i\in P$ with $\overline{i}\not\in P$. Define the \emph{weight} $w(P)$ of $P$ to be the number of pairs $\{j,\overline{j}\}$ contained in $P$, plus one if the defect is non-zero. For example the sets $\{1,\overline{1},2,\overline{2}\}$, $\{1,2,\overline{2}\}$ and $\{1,2,3,4,5,\overline{5}\}$ all have weight two (and defect zero, one and four, respectively). Also note that $P=E(*)$ itself is positive and has defect zero and weight $n$. Let $\aie(P;pos)$ be the subcomplex of $\aie_n(pos)$ supported on those $0$-simplices $A$ such that $A\subseteq P$. We now claim that $\aie(P;pos)$ is $(w(P)-3)$-connected, so $\aie_n(pos)$ being $(n-3)$-connected is a special case of this.
 
 We induct on $w(P)$. For the base case we can use $w(P)=1$, and the result holds vacuously since every set is $(-2)$-connected. Now let $w(P)\ge 2$. Let $D$ be the set of all $i\in P$ with $\overline{i}\not\in P$, so $d(P)=|D|$. Within this induction on $w(P)$ we now additionally begin an induction on $d(P)$. For the base case we assume $D=\emptyset$, i.e., $d(P)=0$. Consider the $0$-simplices of $\aie(P;pos)$ of the form $P\setminus \{\overline{i}\}$, for each $i\in P\cap [n]$. Call these \emph{hubs}, and denote $P\setminus\{\overline{i}\}$ by $\Theta_i$. Note that a symmetric ideal edge in $\aie(P;pos)$ is compatible with a given hub if and only if it is contained in it (i.e., it cannot properly contain it nor be disjoint from it). Any collection of pairwise compatible symmetric ideal edges in $\aie(P;pos)$ lies in the star of some hub, so $\aie(P;pos)$ is covered by the contractible stars of the $\Theta_i$. The intersection of the stars of any $k$ hubs, say $\Theta_{i_1},\dots,\Theta_{i_k}$, is isomorphic to the complex $\aie(P\setminus\{\overline{i}_1,\dots,\overline{i}_k\};pos)$. For $k=1$ this is contractible, being a star, and for $k>1$ we have $w(P\setminus\{\overline{i}_1,\dots,\overline{i}_k\}) = w(P)-k+1 < w(P)$, so by induction on $w(P)$ we know this is $(w(P)-k-2)$-connected. Finally, the nerve of the covering of $\aie(P;pos)$ by these stars is the boundary of a $(w(P)-1)$-simplex, i.e., a $(w(P)-2)$-sphere, so by the first statement in Lemma~\ref{lem:nerve} we conclude that $\aie(P;pos)$ is $(w(P)-3)$-connected.

Now suppose $D\ne\emptyset$, so $d(P)>0$. Without loss of generality we can write $D=\{1,\dots,d\}$. We will build up to $\aie(P;pos)$ from a subcomplex with a known homotopy type, namely the contractible star of the $0$-simplex $\{1\}\cup (P\setminus D)$. The $0$-simplices of $\aie(P;pos)$ missing from this star are those $A$ containing an element of $\{2,\dots,d\}$ (if $d=1$ there is nothing to do, so assume $d\ge 2$), so to obtain $\aie(P;pos)$ from this star we will attach these missing $0$-simplices, in some order, along their relative links $\lk_{rel} A$. If we can do this in an order such that the relative links are always $(w(P)-4)$-connected, then we can conclude that $\aie(P;pos)$ is $(w(P)-3)$-connected. The order is as follows: first glue in the $A$ containing $2$ in order of decreasing size, then the $A$ containing $3$ in order of decreasing size, and so forth. The relative link $\lk_{rel} A$ of $A$ decomposes into the join of its \emph{relative in-link} $\lk_{rel}^{in} A$ and \emph{relative out-link} $\lk_{rel}^{out}A$. The relative in-link of $A$ is defined to be the subcomplex supported on those $B$ in $\lk_{rel} A$ such that $B\subseteq A$. The relative out-link is defined to be the subcomplex supported on those $B$ in $\lk_{rel} A$ that satisfy either $A\subseteq B$ or $A\cap B = \emptyset$. These options encompass all the ways a symmetric ideal edge can be compatible with $A$, and clearly everything in $\lk_{rel}^{in} A$ is compatible with everything in $\lk_{rel}^{out} A$, so indeed $\lk_{rel} A = \lk_{rel}^{in} A * \lk_{rel}^{out} A$. To show that $\lk_{rel} A$ is $(w(P)-4)$-connected for every $A$ containing an element of $\{2,\dots,d\}$, we will consider $\lk_{rel}^{in} A$ and $\lk_{rel}^{out} A$ separately. Let $\{i_A\}\defeq A\cap D$, so $i_A\in\{2,\dots,d\}$, and let $A_\flat\defeq A\setminus \{i_A\}$. The $0$-simplices $B$ in $\lk_{rel}^{in} A$ must lie in $\aie(A_\flat;pos)$, since for $B$ to come before $A$ in our order while having smaller cardinality than $A$, it must not contain $i_A$ (so such $B$ are actually already in the star of $\{1\}\cup (P\setminus D)$). Hence $\lk_{rel}^{in} A$ is isomorphic to $\aie(A_\flat;pos)$, and $w(A_\flat)=w(A)-1<w(P)$, so by induction on $w(P)$ we know $\lk_{rel}^{in} A$ is $(w(A_\flat)-3)$-connected. Next, the $0$-simplices $B$ in $\lk_{rel}^{out} A$ must be disjoint from $\{i_A + 1,\dots,d\}$ and either properly contain $A$ or be disjoint from $A$. The map $B\mapsto B\setminus A_\flat$ induces an isomorphism from $\lk_{rel}^{out} A$ to $\aie(P\setminus (A_\flat \cup \{i_A + 1,\dots,d\});pos)$; the inverse map sends $C$ to itself if $i_A \not\in C$ and to $C\cup A_\flat$ if $i_A\in C$. Since $w(P\setminus (A_\flat \cup \{i_A + 1,\dots,d\})) = w(P)-w(A_\flat)<w(P)$, by induction on $w(P)$ we know $\lk_{rel}^{out} A$ is $(w(P)-w(A_\flat)-3)$-connected. We conclude that $\lk_{rel} A = \lk_{rel}^{in} A * \lk_{rel}^{out} A$ is $(((w(A_\flat)-3) + (w(P)-w(A_\flat)-3)) +2)$-connected, which is to say $(w(P)-4)$-connected, as desired.

This finishes the inductive proof, which in particular shows $\aie_n(pos)$ is $(n-3)$-connected. Now to see that it is not $(n-2)$-acyclic, consider the covering of $\aie_n(pos)$ by the stars of $\Theta_1,\dots,\Theta_n$, as above. The intersection of any $k$ of these stars is $(n-k-2)$-connected, as was deduced during the inductive proof, and the nerve of the covering is an $(n-2)$-sphere, so Lemma~\ref{lem:nerve} says $\aie(P;pos)$ is not $(n-2)$-acyclic.
\end{proof}

A parallel argument shows that $\aie_n(\chi)$ is also $(n-3)$-connected but not $(n-2)$-acyclic for $\chi$ a negative character of $\pgroup_n$.

\begin{remark}
 If $\chi$ is neither positive nor negative then $\aie_n(\chi)$ is much more complicated, for example as discussed in Remark~\ref{rmk:n>3} below, one can find examples of generic $\chi$ for which $\aie_4(\chi)$ has non-trivial $\pi_1$ and $H_2$. Hence, we have focused only on positive and negative characters of $\pgroup_n$ in Theorem~A, but in Subsection~\ref{sec:n3} we will show that generic characters are also tractable at least when $n=3$.
\end{remark}

\medskip

We can now prove Theorem~A.

\begin{proof}[Proof of Theorem~A]
 By Corollary~\ref{cor:alk_non_rose_cible} and Proposition~\ref{prop:pos_alk_conn}, all the ascending links of $0$-simplices in $\cpx_n$ are $(n-3)$-connected, so Corollary~\ref{cor:morse} says the filtration $(\cpx_n^{\chi \ge t})_{t\in\R}$ is essentially $(n-3)$-connected, and so $[\chi]\in\Sigma^{n-2}(\pgroup_n)$.
 
 To prove the negative statement, note that Proposition~\ref{prop:pos_alk_conn} says that there exist ascending links of $0$-simplices that are not $(n-2)$-acyclic, with arbitrary $h_\chi$ value. Also, every ascending link has trivial $(n-1)$st homology since it is $(n-2)$-dimensional. By Corollary~\ref{cor:morse} then, the filtration $(\cpx_n^{\chi \ge t})_{t\in\R}$ is not essentially $(n-2)$-connected, and so $[\chi]\not\in\Sigma^{n-1}(\pgroup_n)$.
\end{proof}

\begin{remark}
 Using the natural split epimorphisms $\pgroup_n \to \pgroup_m$ for $m<n$, we also now can see that if $\chi=\sum_{i\ne j}a_{i,j}\chi_{i,j}$ is a character of $\pgroup_n$ induced from this epimorphism by a positive or negative character of $\pgroup_m$ (so $a_{i,j}$ is positive for $1\le i,j\le m$ or negative for all $1\le i,j\le m$, and is zero if either $i$ or $j$ is greater than $m$) then $[\chi]\in \Sigma^{m-2}(\pgroup_n)$. However, we cannot immediately tell whether $[\chi]\not\in \Sigma^{m-1}(\pgroup_n)$ (which we suspect is the case), since Pettet showed the kernels of $\pgroup_n \to \pgroup_m$ have bad finiteness properties \cite{pettet10}.
\end{remark}

As an immediate consequence of Theorem~A, Citation~\ref{cit:sig_fin} and Observation~\ref{obs:antipodes}, we have the following result, which will provide the crucial step in proving Theorem~B in Section~\ref{sec:symm_auts}.

\begin{corollary}\label{cor:bb}
 For $n\ge 2$, if $\chi$ is a discrete positive character of $\pgroup_n$, then the kernel $\ker(\chi)$ is of type $\F_{n-2}$ but not $\F_{n-1}$. In particular the ``Bestvina--Brady-esque'' subgroup $BB_n$, i.e., the kernel of the character sending each standard generator $\alpha_{i,j}$ to $1\in\Z$, is of type $\F_{n-2}$ but not $\F_{n-1}$. \qed
\end{corollary}

\subsection{The $n=3$ case}\label{sec:n3}

When $n=3$ the $0$-simplex links in $\cpx_3$ are graphs, and using some graph theoretic considerations we can actually prove the analog of Proposition~\ref{prop:pos_alk_conn} for generic characters, which leads to the following:

\begin{theorem}\label{thrm:n3}
 $\Sigma^2(\pgroup_3)=\emptyset$.
\end{theorem}

\begin{proof}
 Since $\Sigma^2(\pgroup_3)$ is open and the generic character classes are dense in the character sphere (Observation~\ref{obs:generic_dense}), it suffices to prove the analog of Proposition~\ref{prop:pos_alk_conn} for all generic $\chi$. Since $\aie_3(\chi)$ is $1$-dimensional, i.e., a graph, we need to prove that it is connected but not a tree.
 
 First we collect some facts about $\ie_3$. It is a graph with eighteen vertices, namely there are twelve vertices corresponding to the symmetric ideal edges $A\subseteq E(*)=\{1,\overline{1},2,\overline{2},3,\overline{3}\}$ with $|A|=3$ (three choices of which $\{i,\overline{i}\}$ to split, times two choices of which of $i$ or $\overline{i}$ to include in $A$, times two choices of which $\{j,\overline{j}\}$ ($j\ne i$) to include in $A$) and six vertices for the symmetric ideal edges with $|A|=5$ (six choices of which element of $E(*)$ to leave out of $A$). Call the former vertices \emph{depots} and the latter \emph{hubs}. There is an edge connecting a depot to a hub whenever the depot is contained in the hub, and there is an edge connecting two depots whenever they are disjoint. Each depot has degree four and each hub has degree six.
 
 Now, since $\chi$ is generic, for each pair of depots of the form $\{i,j,\overline{j}\}$ and $\{\overline{i},j,\overline{j}\}$, precisely one of them is ascending, with a similar statement for hubs. Hence $\aie_3(\chi)$ is a full subgraph of $\ie_3$ spanned by six depots and three hubs. We claim that each ascending hub has degree at least three in $\aie_3(\chi)$. Consider the hub $\{i',j,\overline{j},k,\overline{k}\}$, where $i'\in\{i,\overline{i}\}$ and $i,j,k$ are distinct. If this is ascending then at least one of the depots $\{i',j,\overline{j}\}$ or $\{i',k,\overline{k}\}$ must be as well, since if $a_{j,i}+a_{k,i}$ is positive (respectively negative) then at least one of $a_{j,i}$ or $a_{k,i}$ must be as well. The other two edges incident to our hub come from the fact that one of the depots $\{j,k,\overline{k}\}$ or $\{\overline{j},k,\overline{k}\}$ is ascending, as is one of $\{j,\overline{j},k\}$ or $\{j,\overline{j},\overline{k}\}$.
 
 Having shown that each hub in $\aie_3(\chi)$ has degree at least three in $\aie_3(\chi)$, this tells us $\aie_3(\chi)$ has at least nine edges, since hubs cannot be adjacent. Then since $\aie_3(\chi)$ also has nine vertices we conclude that it contains a non-trivial cycle. It remains to prove it is connected. Say the three hubs in $\aie_3(\chi)$ are $u$, $v$ and $w$, and suppose that $u$ has no adjacent depots in $\aie_3(\chi)$ in common with either $v$ or $w$ (since otherwise we are done). Since there are only six depots, this implies that $v$ and $w$ have the same set of adjacent depots in $\aie_3(\chi)$. But this is impossible since the intersection of the stars of two different ascending hubs can contain at most two ascending depots.
\end{proof}

Combining this with Orlandi-Korner's computation of $\Sigma^1(\pgroup_3)$, we get in particular that $\Sigma^1(\pgroup_3)$ is dense in $S(\pgroup_3)=S^5$ but $\Sigma^2(\pgroup_3)$ is already empty. (Note since $\pgroup_2 \cong F_2$ we also know that $\Sigma^1(\pgroup_2)=\emptyset$.)

\begin{remark}\label{rmk:n>3}
 Unfortunately the analogous result for arbitrary $n$, i.e., that $\Sigma^{n-2}(\pgroup_n)$ is dense in $S(\pgroup_n)$ but $\Sigma^{n-1}(\pgroup_n)$ is empty, cannot be deduced using our methods when $n>3$ (though we do suspect it is true). One would hope that the analog of Proposition~\ref{prop:pos_alk_conn} always holds for all generic $\chi$, but it does not. For example, when $n=4$ one can find a generic character $\chi$ such that $\aie_4(\chi)$ is not simply connected. One example we found is to take $\chi$ with $a_{1,2} = a_{2,1} = a_{3,4} = a_{4,3} = 3$ and all other $a_{i,j} = -1$ (adjusted slightly by some tiny $\varepsilon>0$ to be generic). This non-simply connected ascending link also has non-trivial second homology, so this does not necessarily mean that $[\chi]$ is not in $\Sigma^2(\pgroup_4)$ (and we believe that it actually is), it is just inconclusive. In general we tentatively conjecture that $\Sigma^{n-2}(\pgroup_n)$ is dense in $S(\pgroup_n)$ and $\Sigma^{n-1}(\pgroup_n)$ is empty, but for now our Morse theoretic approach here seems to only be able to handle the positive and negative characters for arbitrary $n$, and also the generic characters for $n=3$.
\end{remark}


\section{Proof of Theorem~B}\label{sec:symm_auts}

We can now use our results about $\pgroup_n$ to quickly prove Theorem~B, about $\group_n$.

\begin{proof}[Proof of Theorem~B]
 Since $BB_n$ is of type $\F_{n-2}$ but not $\F_{n-1}$ by Corollary~\ref{cor:bb}, and has finite index in $\group_n'$ by Lemma~\ref{lem:bb_decomp}, we know that $\group_n'$ is of type $\F_{n-2}$ but not $\F_{n-1}$. Also, Observation~\ref{obs:antipodes} says that for any $m$ either $\Sigma^m(\group_n)$ is all of $S^0$ or else is empty. The result now follows from Citation~\ref{cit:sig_fin}.
\end{proof}

\bibliographystyle{amsalpha}
\providecommand{\bysame}{\leavevmode\hbox to3em{\hrulefill}\thinspace}
\providecommand{\MR}{\relax\ifhmode\unskip\space\fi MR }
\providecommand{\MRhref}[2]{%
  \href{http://www.ams.org/mathscinet-getitem?mr=#1}{#2}
}
\providecommand{\href}[2]{#2}

\end{document}